\documentclass[a4paper,reqno]{amsart}
\usepackage[dvipdfmx]
{color,graphicx}
\usepackage{amsfonts,amsthm,amssymb,bm,mathrsfs,nccmath}
\usepackage{mathtools}
\mathtoolsset{showonlyrefs=true}
\usepackage[dvipdfmx]
{hyperref}
\usepackage{hyperref}
\newtheoremstyle{mystyle}
    {1em}{}{\it}{}{\bf}{.}{12pt}{}
\theoremstyle{mystyle} %

\newtheorem{theorem}{Theorem}[section]

\newtheorem{lemma}[theorem]{Lemma}
\renewcommand{\proofname}{Proof}
\newtheorem{prop}[theorem]{Proposition}

\makeatletter 
\renewenvironment{proof}[1][\proofname]{\par
  \pushQED{\qed}
  \normalfont \topsep6\p@\@plus6\p@\relax
  \trivlist
  \item[\hskip\labelsep
        \itshape
     {\bf{#1}\@addpunct{.}\ }]
}{\popQED\endtrivlist\@endpefalse}
\makeatother

\usepackage{chngcntr}
\hypersetup{%
 bookmarksnumbered=true,%
 bookmarksopen=true,%
 colorlinks=true,%
 linkcolor=blue,
 citecolor=blue,
}
\counterwithin{equation}{section}

\def\({\left(}
\def\){\right)}
\def\red{
}

\makeatletter
\g@addto@macro\bfseries{\boldmath}
\makeatother
\allowdisplaybreaks[1]
\begin{document}
\title[Sharp Uncertainty Principle inequality for solenoidal fields]{Sharp Uncertainty Principle inequality\\ for solenoidal fields}
\author[N. Hamamoto]{Naoki Hamamoto}
\address{\parbox{\linewidth}{Department of Mathematical Sciences, Osaka Prefecture University, 
 \\
Sakai, Osaka 599-8531, Japan
}}
\email{yhjyoe@yahoo.co.jp 
{\rm (N. Hamamoto)}}
\begin{abstract}
This paper solves the $L^2$ version of Maz'ya's open problem \cite[Section~3.9]{Mazya_75} on the sharp uncertainty principle inequality
\[ \int_{\mathbb{R}^N}|\nabla {\bm u}|^2dx\int_{\mathbb{R}^N}|{\bm u}|^2|{\bm x}|^2dx\ge C_N\(\int_{\mathbb{R}^N}|{\bm u}|^2dx\)^2\]
 for solenoidal (namely divergence-free) vector fields ${\bm u}={\bm u}({\bm x})$ 
on $\mathbb{R}^N$. 
The best value of the constant turns out to be $C_N=\frac{1}{4}\(\sqrt{(N-2)^2+8}+2\)^2$ which exceeds the original value $N^2/4$ for unconstrained fields. Moreover, we show the attainability of $C_N$ and specify the profiles of the extremal solenoidal fields: for $N\ge4$, the extremals are proportional to a poloidal field that is axisymmetric and unique up to the invariant transformation the axis of symmetry \red{and the scaling transformation}; for $N=3$, there additionally exist extremal toroidal fields.  
\end{abstract}
\subjclass[2010]{26D10 (Primary);  35A23, 26D15, 81S07 (Secondary)}
\keywords{Uncertainty Principle inequality, Solenoidal, Poloidal, Toroidal, Spherical harmonics, Best constant, Laguerre polynomials, Confluent hypergeometric function of the first kind}
\maketitle
\section{Introduction}
We study the functional inequality called the  Heisenberg's Uncertainty Principle inequality for real (not complex) vector fields on $\mathbb{R}^N$,  with focus on how the best constant is changed by imposing differential constraints on the test vector fields.
\subsection{Basic notations}
\label{subsec:BN}
Throughout this paper, we use bold letters to denote vectors in the $N$-dimensional Euclidean space, e.g. ${\bm x}=(x_1,x_2,\cdots,x_N)\in\mathbb{R}^N$. 
By writing ${\bm u}=(u_1,u_2,\cdots,u_N)\in C^\infty(\Omega)^N$, we mean that\[{\bm u}:\Omega\to\mathbb{R}^N,\qquad {\bm x}\mapsto {\bm u}({\bm x})=(u_1({\bm x}),\cdots,u_N({\bm x}))\]
is a vector field on $\Omega$ with the components $\{u_1,\cdots,u_N\}\subset C^\infty(\Omega)$; the same also applies to other spaces: e.g., ${\bm u}\in C_c^\infty(\Omega)^N$ means that $\{u_1,\cdots,u_N\}\subset C_c^\infty(\Omega)$, i.e., ${\bm u}$ is a smooth vector field with compact support on $\Omega$. 
We use the gradient operator $\nabla=\big(\frac{\partial}{\partial x_1},\cdots,\frac{\partial}{\partial x_N}\big)$ not only on scalar fields but also vector fields ${\bm u}$: we view $\nabla {\bm u}=\big(\frac{\partial u_j}{\partial x_k}\big)_{j,k\in\{1,\cdots,N\}}:\Omega\to\mathbb{R}^{N\times N}$ as a matrix field. The notation ${{\bm x}\cdot {\bm y}=\sum_{k=1}^Nx_ky_k}$ denotes the standard scalar product of two vectors, which induces $|{\bm x}|=\sqrt{{\bm x}\cdot {\bm x}}$ the absolute value of ${\bm x}$; the same also applies to matrix fields, e.g.,  \[\nabla {\bm u}\cdot\nabla {\bm v}=\sum_{j,k}\frac{\partial u_j}{\partial x_k}\frac{\partial v_j}{\partial x_k}\quad\text{ and }\quad |\nabla {\bm u}|^2=\sum_{j,k}\Big(\frac{\partial u_j}{\partial x_k}\Big)^2.\] 
\subsection{Motivation and related known results}
Heisenberg's Uncertainty Principle \cite{Heisenberg_1927} states in quantum mechanics that the observed values of the position and momentum of a particle cannot be determined at the same time. This was mathematically formulated as an inequality called the Heisenberg-Pauli-Weyl Uncertainty Principle (or shortly HUP, see \cite{Weyl_1931}), which asserts that a wave function and its Fourier transform cannot be simultaneously sharply localized at a single point. 
In the present paper,  we employ the expression of HUP inequality in terms of real vector fields on $\mathbb{R}^N$, which is precisely described as follows:
\begin{equation}
 \int_{\mathbb{R}^N}|\nabla {\bm u}|^2dx \int_{\mathbb{R}^N}|{\bm u}|^2|{\bm x}|^2dx\ge C_N
  \(\int_{\mathbb{R}^N}|{\bm u}|^2dx\)^2
 \label{UP0}
\end{equation}
for ${\bm u}={\bm u}({\bm x})\in C^\infty(\mathbb{R}^N)^N$ with a suitable integrability condition. 
Here the constant 
\begin{equation}
 C_N=N^2/4
\label{CN_un}
\end{equation}
is known to be sharp for unconstrained vector fields ${\bm u}$ and attained when the components $\{u_1,\cdots,u_N\}$ have the Gaussian profile, that is, each of them is proportional to $e^{-c|{\bm x}|^2}$ for some constant $c>0$ (see e.g. \cite{Folland_UP}). Note that the inequality \eqref{UP0} is equivalent to its scalar-field version, since the former easily holds by applying the latter to each of the components of ${\bm u}$ with the aid of Cauchy-Schwarz inequality; for the details, see the discussion in \cite{CFL_HUP}. 

Now, a non-trivial problem occurs when ${\bm u}$ is subject to some differential constraint,  asking how much is the new best value of $C_N$ larger than \eqref{CN_un}.  
When ${\bm u}$ is assumed to be curl-free, namely when ${\bm u}=\nabla \phi$ holds with some scalar potential $\phi$, the answer to the problem was recently obtained by Cazacu-Flynn-Lam  \cite{CFL_HUP}:  
it was shown for all $N\ge1$ that the inequality \eqref{UP0} holds with the sharp constant  
\begin{equation}
C_N=\frac{1}{4}(N+2)^2
\label{CN_CFL}
\end{equation} 
for  curl-free fields ${\bm u}$, 
 and that the sharp value is attained when  $\phi$ has the Gaussian profile.  
This result gives a curl-free improvement of the HUP inequality, in the sense that the best value of $C_N$ gets larger by restricting the test vector fields to curl-free fields. 
In the case $N=2$, as is well known, the curl-free fields are isometrically isomorphic to solenoidal (namely divergence-free) vector fields. Hence the result of Cazacu-Flynn-Lam  also solves the problem of finding the best value of $C_2$ for solenoidal fields, as a special case of the question asked by Maz'ya in the $L^2$ setting which reads as follows: 

\

\noindent
{\bf Open Problem} (Maz'ya \cite[Section~3.9]{Mazya_75}).   
{\it Find the new best value of the constant $C_N$ in the inequality \eqref{UP0} %
 when ${\bm u}$ is assumed to be solenoidal.
} 

\

This problem has attracted a number of mathematicians' interest and has an exotic feel in the following sense: the uncertainty principle is a concept of quantum mechanics, while solenoidal vector fields are objects of fluid dynamics or electromagnetism; the two are quite different from each other, and the solenoidal condition is usually not considered as a physical property of wave functions of particles in quantum mechanics; nevertheless, in mathematics, the above open problem does not hesitate to connect the two things, which could also be expected to bring a new unknown physical meaning.

From the viewpoint of mathematical analysis, the open problem is considered as a solenoidal improvement of the HUP inequality. 
Historically, the solenoidal improvement of functional inequalities probably first appears in the article \cite{Costin-Mazya} by Costin-Maz'ya, who derived the sharp Hardy inequality
\[
 \int_{\mathbb{R}^N}|\nabla {\bm u}|^2dx\ge \(\frac{N-2}{2}\)^2\frac{(N+2)^2}{(N+2)^2-8}\int_{\mathbb{R}^N}\frac{|{\bm u}|^2}{|{\bm x}|^2}dx
\]
for axisymmetric solenoidal fields ${\bm u}$; here the constant $\(\frac{N-2}{2}\)^2\frac{(N+2)^2}{(N+2)^2-8}$ is sharp, which improves the original best value $\(\frac{N-2}{2}\)^2$ in the same inequality (for unconstrained fields) found by Leray \cite{Leray}  for $N\ge3$ and Hardy \cite{Hardy} for $N=1$. Concerning such an improvement, Costin-Maz'ya's result was further refined in a series of recent papers \cite{swirl_CPAA,3D_NA,HL_pre,HL-const_AdM}; in particular,  the removal of the axisymmetry condition on the test solenoidal fields was achieved without changing the best constant, which also includes a solution to Maz'ya's another open problem \cite[Section~9.4]{Mazya_75}. In addition, other inequalities including Rellich inequality were found in \cite{CF_MAAN,CF_JFA,RL_div,RL_CVPD,CF-RH_pre} with the new best constant for curl-free or solenoidal fields. (Incidentally, see also \cite{CF_MIA} for the two-dimensional logarithmic weighted Hardy inequality.) 
In any case,  the best constants in such inequalities have turned out to be computable for solenoidal fields on $\mathbb{R}^N$. On the other hand, as mentioned in e.g. \cite{Brezis-Vazquez},  Hardy's inequality can also be considered as another kind of HUP; to see a rough picture of it, suppose that ${\bm u}({\bm x})$ is sharply localized at ${\bm x}={\bm 0}$, then the integral $\int_{\mathbb{R}^N}\frac{|{\bm u}|^2}{|{\bm x}|^2}dx$ and hence $\int_{\mathbb{R}^N}|\nabla {\bm u}|^2dx$ gets larger. Then we may expect that the same computability also applies to the best constant in HUP inequality (henceforth, ``best HUP constant'' for short). 

\subsection{Main result}
Motivated by the observation above, we try to solve the aforementioned open problem for $N\ge3$. 
We always assume that the test solenoidal fields ${\bm u}$ are regular, in the sense that the three integrals in the HUP inequality are finite,  which we express as the finiteness of the norm%
\begin{equation}
\begin{aligned}
 \|{\bm u}\|_{\mathcal{D}}&:=\(\int_{\mathbb{R}^N}|{\bm u}|^2dx+\int_{\mathbb{R}^N}|{\bm u}|^2|{\bm x}|^2dx+\int_{\mathbb{R}^N}|\nabla {\bm u}|^2dx\)^{1/2}.
\end{aligned}
\label{norm_D}
\end{equation}
We denote by $\mathcal{D}(\mathbb{R}^N)$ the Hilbert space completion of $C_c^\infty(\mathbb{R}^N)^N$ with respect to the norm $\|\cdot\|_{\mathcal{D}}$. 
Now, our main result reads as follows: %
\begin{theorem}
\label{theorem}
Let ${\bm u}\in \mathcal{D}(\mathbb{R}^N)$ be a solenoidal field. %
Then the inequality \eqref{UP0} holds with the best constant 
\[
 C_N=\frac{1}{4}\(\sqrt{(N-2)^2+8}+2\)^2.
\]
Moreover, $C_N$ is attained in $\mathcal{D}(\mathbb{R}^N)\setminus\{{\bm 0}\}$. 
\end{theorem}
{\it Remark.} In Section \ref{sec:concl}, the solenoidal fields ${\bm u}\not\equiv {\bm 0}$ satisfying the equality sign in \eqref{UP0}, which we say {\it extremal}, will be specified; as a result, they are classified into two profiles when $N=3$ and only one when $N\ge4$.  
																																																	For both the cases, there exist extremal poloidal fields (in the sense of the so-called poloidal-toroidal fields in Section \ref{sec:VC}) that are proportional to an axisymmetric solenoidal field uniquely determined up to the axis of rotation and scaling transform (see \eqref{extremal_pol}). For $N=3$, there also exist extremal toroidal fields (see \eqref{extremal_tor}). 

In a comparison between the three HUP best constants shown above for unconstrained, solenoidal or curl-free fields, it is easy to check that 
\[
 \frac{1}{4}N^2<\frac{1}{4}\(\sqrt{(N-2)^2+8}+2\)^2\le \frac{1}{4}(N+2)^2, 
\]
or that the value of  $C_N$ in Theorem~\ref{theorem} stays between \eqref{CN_un} and \eqref{CN_CFL}. This result agrees with that the solenoidal condition is weaker than the curl-free one in view of HUP, in the sense that the solenoidal condition consists of only one differential equation while the curl-free condition requires many (or exactly $\frac{1}{2}N(N-1)$) equations. 
\subsection{Overview of this paper}
In the rest of the present paper, we focus on the proof of Theorem~\ref{theorem} %
that consists of the following five parts:
\begin{itemize}
 \item Section \ref{sec:VC}. Poloidal-toroidal (or shortly PT) decomposition ${\bm u}={\bm u}_P+{\bm u}_T$ of solenoidal fields ${\bm u}$.
\item Section \ref{sec:tor}. Computation of the best HUP constant for toroidal fields ${\bm u}_T$.
\item Section \ref{sec:pol}. Integral calculation for spherical harmonics components of ${\bm u}_P$.
\item Section \ref{sec:Laguerre}. One-dimensional minimization problem for the functional of the form
\[ R [g]=
      \frac{\int_0^\infty
      (\red{g''})^2x^{\mu+1}dx\int_0^\infty
\({x}^2(\red{g'})^2  -\varepsilon g^2\){{x}}^{\mu-1}d{x}}
      {\int_0^\infty
      (\red{g'})^2 {{x}}^{\mu }d{x}}\]
on the set of test functions $g=g(x):[0,\infty)\to\mathbb{R}$ with a suitable regularity condition, where $\mu$ and $\varepsilon$ are positive parameters in some ranges.
\item Section \ref{sec:concl}. Conclusion.
\end{itemize}

The PT decomposition theorem in Section \ref{sec:VC} historically originates from  G.~Backus \cite{Backus} for $N=3$ and was generalized by N.~Weck \cite{Weck} to differential forms on $\mathbb{R}^N$. We use their results in the framework of the standard vector calculus (instead of differential forms) to separate the calculation of $C_N$ into two computable parts, namely the best HUP constants $C_{P,N}$ and $C_{T,N}$ for poloidal and toroidal fields, respectively; the same technique was also used in the works \cite{3D_NA,HL_pre,RL_CVPD} on sharp Hardy or Rellich inequality for solenoidal fields. 

The computation of $C_{T,N}$ in Section \ref{sec:tor} will be carried out by reducing it to the problem of finding the best HUP constant for scalar fields with zero spherical mean. 

The calculation of $C_{P,N}$ in Section \ref{sec:pol} will be carried out by estimating the integrals in \eqref{UP0} for ${\bm u}={\bm u}_P$ along the idea of Cazacu-Flynn-Lam \cite{CFL_HUP} that originates from Tertikas-Zographopoulos \cite{Tertikas-Z}, who used the transformation of type $\phi({\bm x})\mapsto |{\bm x}|^{\lambda}\phi({\bm x})$ with the exponent $\lambda$ varying corresponding to the spherical harmonics components of test fields. By way of this procedure, the estimation for $C_{P,N}$ can be reduced to one-dimensional minimization problem.

\red{The main difficulty lies in Section \ref{sec:Laguerre} where we try to solve the minimization problem for the aforementioned functional $R[g]$. Problems of the same or more general type also appear in recent papers (e.g. \cite{CFL_HUP,CKN_second}), which seem to remain unsolved. 
To overcome the difficulty in our case, 
we make use of a simple calculation method for the $L^2$ integral \eqref{I[g]} of Kummer's equation.  
However, it should also be noted that the discovery of such a method is based on a long and complicated computation of $R[g]$, carried out in the previous version of the preprint \cite{HUP_pre} by expressing $g$ in the (general) Laguerre polynomial expansion, where we succeeded in finding that the extremal function of $R[g]$ must be of Kummer type. 
\section{Preliminary: poloidal-toroidal decomposition 
} 
\label{sec:VC}
After preparing basic notations (in addition to \ref{subsec:BN}), we briefly review poloidal-toroidal fields in the framework of the standard vector calculus on $\dot{\mathbb{R}}^N$, as a specialization of Weck's work \cite{Weck}. We omit the proofs of the elementary facts related to the PT fields, since they have already been fully discussed in the previous papers; see \cite{HL_pre,RL_CVPD,CF_JFA} for the details. By using the PT decomposition theorem,  the best HUP constant can be separated into two computable parts, as a consequence of the $L^2(\mathbb{S}^{N-1})$ orthogonality of PT fields. 
\subsection{Basic notations}
\label{subsec:notation}
We make the identification $\dot{\mathbb{R}}^N\cong \mathbb{R}_+\times\mathbb{S}^{N-1}$, in the sense that $\dot{\mathbb{R}}^N=\mathbb{R}^N\setminus\{{\bm 0}\}$ is a smooth manifold diffeomorphic to  the product of the $(N-1)$-sphere $\mathbb{S}^{N-1}=\left\{{\bm x}\in\mathbb{R}^N\ ;\ |{\bm x}|=1\right\}$ and the open half line $\mathbb{R}_+=\left\{{r}\in\mathbb{R}\ ;\ {r}>0\right\}$, and that every point ${\bm x}$ in $\dot{\mathbb{R}}^N$ is uniquely expressed as ${\bm x}={r}{\bm \sigma}$ in terms of the radius (namely radial  coordinate) ${r}=|{\bm x}|>0$ and the unit vector ${\bm \sigma}={\bm x}/|{\bm x}|\in\mathbb{S}^{N-1}$. 

For every vector field ${\bm u}%
={\bm u}({\bm x}):\dot{\mathbb{R}}^N\to \mathbb{R}^N$,
there exists an unique pair of a scalar field $u_R$ and a vector field ${\bm u}_S$ satisfying
\begin{equation}
 {\bm u}= {\bm \sigma}u_R+{\bm u}_S\quad\text{ and }\quad {\bm \sigma}\cdot{\bm u}_S=0\qquad\text{on }\dot{\mathbb{R}}^N,
\label{RS}  
\end{equation}
and these fields have the explicit expressions $u_R={\bm \sigma}\cdot {\bm u}$ and ${\bm u}_S={\bm u}-{\bm \sigma}u_R$ 
which we call the radial component and the spherical part of ${\bm u}$, respectively. 

The gradient operator $\nabla=\big(\frac{\partial}{\partial x_1},\cdots,\frac{\partial}{\partial x_N}\big)$ and the Laplacian $\triangle=\sum_{k=1}^N (\frac{\partial}{\partial x_k})^2$ can be decomposed into radial-spherical parts as
\begin{equation}
\label{grad-Lap}
  \nabla =  {\bm \sigma} \partial_{r} + \frac{1\,}{{r}\,}\nabla_{\!\sigma}
\quad\text{ and }\quad
\triangle=\partial_{r}'\partial_{r}+\frac{1}{{r}^2}\triangle_\sigma,
\end{equation}
where $\nabla_{\!\sigma}$ and $\triangle_\sigma$ denote the spherical gradient and the spherical Laplacian (or Laplace-Beltrami operator on $\mathbb{S}^{N-1}$), respectively, and where 
\begin{equation}
\partial_{r}:={\bm \sigma}\cdot\nabla=\sum_{k=1}^N \frac{x_k}{|{\bm x}|}\frac{\partial}{\partial x_k}
 \quad\text{ and }\quad
 \partial_{r}'=\partial_{r}+\frac{N-1}{{r}}
\end{equation}
are (in this order) the radial derivative and its skew $L^2$ adjoint, in the sense that
\[
 \int_{\mathbb{R}^N}f\partial_{r}g\hspace{0.1em}dx=-\int_{\mathbb{R}^N}(\partial_{r}'f)g\hspace{0.1em}dx
 \qquad\forall f,g\in C^\infty(\dot{\mathbb{R}}^N).
\]
Applying the gradient formula in \eqref{grad-Lap} to a vector field ${\bm u}$ in \eqref{RS} and taking the trace part of the matrix field $\nabla {\bm u}$, one can check that 
\begin{equation}
 {\rm div}\hspace{0.1em}{\bm u}=\sum_{K=1}^N \frac{\partial u_k}{\partial x_k}=\partial_{r}'u_R+\nabla_{\!\sigma}\cdot{\bm u}_S,
\end{equation}
where $\nabla_{\!\sigma}\cdot{\bm u}_S$ denotes the trace part of the matrix field $\nabla_{\!\sigma}{\bm u}_S$.
By using this identity, we further obtain  $\nabla_{\!\sigma}\cdot \nabla_{\!\sigma}f=\triangle_\sigma f$ and the spherical integration by parts formula
\begin{equation}
 \int_{\mathbb{S}^{N-1}}{\bm u}\cdot \nabla_{\!\sigma}f\,\mathrm{d}\sigma=-\int_{\mathbb{S}^{N-1}}(\nabla_{\!\sigma}\cdot {\bm u}_S)f\,\mathrm{d}\sigma\qquad\forall f\in C^\infty(\mathbb{S}^{N-1})
\end{equation}
for any fixed radius (see e.g. \cite[Lemma~2.1]{HL_pre}). 
The commutation relations between $\triangle_\sigma$ and ${\bm \sigma}$ or $\nabla_{\!\sigma}$ are given by the following identities:%
\begin{equation}
\left\{\begin{array}{l}
  \triangle_\sigma ({\bm \sigma}f)-{\bm \sigma}\triangle_\sigma f=-(N-1){\bm \sigma}f+2\nabla_{\!\sigma}f,
\vspace{0.5em} \\
 \triangle_\sigma\nabla_{\!\sigma}f-\nabla_{\!\sigma}\triangle_\sigma f=-2\hspace{0.1em}{\bm \sigma}\triangle_\sigma f+(N-3)\nabla_{\!\sigma}f.
\end{array}\right.
\label{comm}
\end{equation}
(For the details of the proof, see e.g., \cite[Lemma~3]{CF_JFA} or \cite[Lemma~7]{CF_MAAN}.)
\subsection{Poloidal-toroidal fields}
\label{subsec:PT}
We say that a vector field ${\bm u}\in C^\infty(\dot{\mathbb{R}}^N)^N$ is {\it pre-poloidal} when 
\[{\bm u}_S=\nabla_{\!\sigma}f\quad\text{on }\dot{\mathbb{R}}^N\quad \text{for some }f\in C^\infty(\dot{\mathbb{R}}^N),\]
and {\it toroidal} when $u_R={\rm div}\,{\bm u}=0$ \ on $\dot{\mathbb{R}}^N$. We denote by $\mathcal{P}(\dot{\mathbb{R}}^N)$ resp. $\mathcal{T}(\dot{\mathbb{R}}^N)$ the set  of all pre-poloidal resp. toroidal fields.  
The principal properties of them are summarized as follows:
\begin{prop}
 \label{prop:prePT}
Let ${\bm v}\in\mathcal{P}(\dot{\mathbb{R}}^N)$ and ${\bm w}\in\mathcal{T}(\dot{\mathbb{R}}^N)$. We abbreviate as ${\bm v}={\bm v}({r}{\bm \sigma})$ and ${\bm w}({r}{\bm \sigma})$ for any ${r}>0$ and ${\bm \sigma}\in\mathbb{S}^{N-1}$.  Then it holds that
\[\begin{split}
&  \int_{\mathbb{S}^{N-1}}{\bm v}\cdot {\bm w}\,\mathrm{d}\sigma=\int_{\mathbb{S}^{N-1}}\nabla {\bm v}\cdot\nabla {\bm w}\,\mathrm{d}\sigma=0\qquad(\forall{r}>0),
\\
   & \int_{\mathbb{S}^{N-1}}{\bm w}\,\mathrm{d}\sigma=0\qquad\(\forall{r}>0%
\),
   \\
   &\left\{\zeta {\bm v},\partial_{r}{\bm v},\triangle_\sigma {\bm v}\right\}\subset\mathcal{P}(\dot{\mathbb{R}}^N)\quad\text{ and }\quad
   \left\{\zeta {\bm w},\partial_{r}{\bm w},\triangle_\sigma {\bm w}\right\}\subset\mathcal{T}(\dot{\mathbb{R}}^N),
  \end{split}
\]
where $\zeta\in C^\infty(\dot{\mathbb{R}}^N)$ is any radially symmetric scalar field. 
\end{prop}
This proposition says that the spaces $\mathcal{P}(\dot{\mathbb{R}}^N)$ and $\mathcal{T}(\dot{\mathbb{R}}^N)$ are $L^2(\mathbb{S}^{N-1})$-orthogonal and invariant under radial (derivative) operators and the Laplacians, and that every toroidal field is of zero spherical mean.

Let us further decompose the space of pre-poloidal fields.
For every $\nu\in\mathbb{N}$, we set \[\alpha_\nu:=\nu(\nu+N-2)\] and
\[\begin{aligned}
   &\mathcal{E}_\nu(\dot{\mathbb{R}}^N):=\left\{f\in C^\infty(\dot{\mathbb{R}}^N)\ ;\ -\triangle_\sigma f=\alpha_\nu f\ \text{ on }\dot{\mathbb{R}}^N\right\},
\\  &\mathcal{P}_\nu(\dot{\mathbb{R}}^N):=\left\{{\bm u}\in \mathcal{P}(\dot{\mathbb{R}}^N)\ :\ {\bm u}={\bm \sigma}f+\nabla_{\!\sigma}g\quad\text{for some } f,g\in \mathcal{E}_\nu(\dot{\mathbb{R}}^N)\right\},
  \end{aligned}
\]
which denote the set of $\nu$-th eigenfunctions of $-\triangle_\sigma$ and that of pre-poloidal fields generated by them.
Note that $\{\mathcal{P}_\nu(\dot{\mathbb{R}}^N)\}_{\nu\in{\mathbb{N}}}$ are invariant under the operation of $\zeta,\partial_{r},\triangle_\sigma$ in the sense of Proposition \ref{prop:prePT}, and that they are $L^2(\mathbb{S}^{N-1})$-orthogonal:
\begin{equation}
 \int_{\mathbb{S}^{N-1}}{\bm u}\cdot {\bm v}\;\!\mathrm{d}\sigma=\int_{\mathbb{S}^{N-1}}\nabla {\bm u}\cdot\nabla {\bm v}\;\!\mathrm{d}\sigma=0\qquad\forall ({\bm u},{\bm v})\in\mathcal{P}_\nu(\dot{\mathbb{R}}^N)\times\mathcal{P}_\rho(\dot{\mathbb{R}}^N)
\label{SH_orth}
\end{equation}
(for any radius ${r}$) holds whenever $\nu\ne\rho$. %

We say that a pre-poloidal field is {\it poloidal} whenever it is solenoidal. 
The {\it poloidal generator} is a second-order differential operator on $C^\infty(\dot{\mathbb{R}}^N)$ given by \[{\bm D}={\bm \sigma}\triangle_\sigma -{r}\partial_{r}'\nabla_{\!\sigma},\] 
which maps every scalar field $f$ to the poloidal field ${\bm D}f\in\mathcal{P}(\dot{\mathbb{R}}^N)$. 
The following fact is fundamental:
\begin{prop}
 \label{prop:PT}
 Let ${\bm u}:\mathbb{R}^N\to \mathbb{R}^N$ be a smooth solenoidal field on $\mathbb{R}^N$. Then there exists an unique pair of poloidal-toroidal fields $({\bm u}_P,{\bm u}_T)\in \mathcal{P}(\dot{\mathbb{R}}^N)\times\mathcal{T}(\dot{\mathbb{R}}^N)$ satisfying
 \[
 {\bm u}={\bm u}_P+{\bm u}_T\qquad\text{\rm on}\ \dot{\mathbb{R}}^N. \]
In particular, the poloidal part has the expression 
\[
 {\bm u}_P={\bm D}f.%
 \]
Here the scalar field $f=\triangle_\sigma^{-1}u_R$ is an unique solution to the Poisson-Beltrami equation $\triangle_\sigma f=u_R$ \ (on $\dot{\mathbb{R}}^N$) under the condition $\int_{\mathbb{S}^{N-1}}f({r}{\bm \sigma})\mathrm{d}\sigma=0$, $\forall{r}>0$.
\end{prop}
In this proposition, we call $f=\triangle_\sigma^{-1} u_R$ the {\it poloidal potential} of ${\bm u}$. %
When it %
is multiplied by a radially symmetric scalar field $\zeta$,  the deformation of ${\bm D}f$ and $\nabla {\bm D}f$ has the following $L^2(\mathbb{S}^{N-1})$ estimates \cite[Lemma~2.3]{RL_CVPD}:
\begin{equation}
 \left.\begin{aligned}  &C\int_{\mathbb{S}^{N-1}}|{\bm D}(\zeta f)-\zeta {\bm D}f|^2\mathrm{d}\sigma  \le (\partial\zeta)^2\int_{\mathbb{S}^{N-1}} |{\bm D}f|^2\mathrm{d}\sigma, \\& C \int_{\mathbb{S}^{N-1}}|\nabla {\bm D}{(\zeta f)}-\zeta\nabla {\bm D}f|^2\mathrm{d}\sigma\le\left( (\partial\zeta)^2+(\partial^2\zeta)^2\right) \int_{\mathbb{S}^{N-1}}\frac{|{\bm D}f|^2}{{r^2}}\mathrm{d}\sigma%
 \end{aligned}
\right\}
\label{Df}
\end{equation}
for any ${r}>0$, where we abbreviate as 
\begin{equation}
\red{\partial={r}\partial_{r}}
\label{abb}
\end{equation} 
and where $C$ is a constant number depending only on $N$. 
In particular, this fact yields that every solenoidal field ${\bm u}$ on $\mathbb{R}^N$ can be $L^2$ approximated by those with compact support on $\dot{\mathbb{R}}^N$, %
which we will use in the following form:
\begin{lemma}
\label{lemma:density}
Let %
${\bm u}\in \mathcal{D}(\mathbb{R}^N)$ be a solenoidal field, and let $\zeta_0\in C_c^\infty(\mathbb{R})$ with $\zeta_0(0)=1$. Define $\{{\bm u}_n\}_{n\in\mathbb{N}}\subset C_c^\infty(\dot{\mathbb{R}}^N)^N$ by%
 \[{\bm u}_n={\bm D} (\zeta_n f)+\zeta_n {\bm u}_T\] 
for the poloidal potential $f$ of ${\bm u}_P$, where we set $\zeta_n({\bm x})=\zeta_0\(\frac{1}{n}\log|{\bm x}|\)$ \ $(\forall n\in\mathbb{N})$. 
Then it holds that %
\[
\|{\bm u}_n-{\bm u}\|_{\mathcal{D}}\to 0%
 \qquad{\rm as }\ n\to\infty.
\]
\end{lemma}
\begin{proof}
 It suffices to separately check the cases ${\bm u}={\bm u}_T$ and ${\bm u}={\bm u}_P$.

Let us consider the case ${\bm u}={\bm u}_T$. Since ${\bm u}_n=\zeta_n {\bm u}$, it is clear from the dominated convergence theorem that $\int_{\mathbb{R}^N}|{\bm u}_n-{\bm u}|^2(1+|{\bm x}|^2)dx\to 0$. Moreover, a direct calculation yields
\[
\begin{aligned}
 \int_{\mathbb{R}^N}&|\nabla {\bm u}_n-\zeta_n \nabla {\bm u}|^2dx
=\int_{\mathbb{R}^N}(\partial_{r}\zeta_n)^2\left|{\bm u}\right|^2dx
\\&=\int_{\mathbb{R}^N}(\partial\zeta_n)^2 \frac{|{\bm u}|^2}{|{\bm x}|^2}dx=O(n^{-2})\int_{\mathbb{R}^N}\frac{|{\bm u}|^2}{|{\bm x}|^2}dx.
\end{aligned}
\]
Notice here that the last integral is finite,  since ${\bm u}\in\mathcal{D}(\mathbb{R}^N)$ and $N\ge3$. Therefore, we get
$ \int_{\mathbb{R}^N}|\nabla {\bm u}_n-\zeta_n\nabla {\bm u}|^2dx\to0$.  
Since $\int_{\mathbb{R}^N}|\zeta_n\nabla {\bm u}-\nabla{\bm u}|^2dx\to0$ also holds from the dominated convergence theorem, by way of the $L^2$ triangle inequality we get $\int_{\mathbb{R}^N}|\nabla{\bm u}_n-\nabla {\bm u}|^2dx\to 0$ and hence $\|{\bm u}_n-{\bm u}\|_{\mathcal{D}}\to0$, as desired.

For the case ${\bm u}={\bm u}_P$, it can be written as ${\bm u}={\bm D}f$. 
By using \eqref{Df}, we have
\[
\begin{aligned}
 \int_{\mathbb{R}^N}&\left|{\bm u}_n -\zeta_n {\bm u} \right|^2(1+|{\bm x}|^2)dx
\\&=\int_{\mathbb{R}^N}\left|{\bm D}(\zeta_n f)-\zeta_n {\bm D}f\right|^2(1+|{\bm x}|^2)dx
\\&\le \frac{1}{C}\int_{\mathbb{R}^N}(\partial \zeta_n )^2|{\bm D}f|^2(1+|{\bm x}|^2)dx=O(n^{-2})\int_{\mathbb{R}^N}|{\bm u}|^2(1+|{\bm x}|^2)dx,
\\
 \int_{\mathbb{R}^N}&\left|\nabla{\bm u}_n -\zeta_n\nabla{\bm u} \right|^2dx
 =\int_{\mathbb{R}^N}\left|\nabla {\bm D}(\zeta_n f)-\zeta_n\nabla {\bm D}f\right|^2dx
\\&\le \frac{1}{C} \int_{\mathbb{R}^N}\((\partial\zeta_n)^2+(\partial^2\zeta_n)^2\)\frac{|{\bm D}f|^2}{|{\bm x}|^2}dx
=O(n^{-2})\int_{\mathbb{R}^N}\frac{|{\bm u} |^2}{|{\bm x}|^2}dx.
\end{aligned}
\]
Notice that both the last integrals are finite, for the same reason as the previous case. %
Therefore, since $\zeta_n {\bm u}\to {\bm u}$ and $\zeta_n\nabla {\bm u}\to \nabla{\bm u}$ are dominated convergence, we get $\|{\bm u}_n-{\bm u}\|_{\mathcal{D}}\to0$ by way of the triangle inequality. 
\end{proof}
\subsection{PT decomposition of the best HUP constant%
}
Let ${\bm u}={\bm u}_P+{\bm u}_T$ be the PT decomposition of a solenoidal field ${\bm u}\in\mathcal{D}(\mathbb{R}^N)$. Then we see from the $L^2(\mathbb{S}^{N-1})$ orthogonality of PT fields  (in the sense of Proposition \ref{prop:prePT})  that 
\begin{equation}
\begin{aligned}
&\int_{\mathbb{R}^N}|\nabla {\bm u}|^2dx\int_{\mathbb{R}^N}|{\bm u}|^2|{\bm x}|^2dx
\\ &=\(\int_{\mathbb{R}^N}|\nabla {\bm u}_P|^2dx+\int_{\mathbb{R}^N}|\nabla {\bm u}_T|^2dx\)
  \(\int_{\mathbb{R}^N}|{\bm u}_P|^2|{\bm x}|^2dx+\int_{\mathbb{R}^N}|{\bm u}_T|^2|{\bm x}|^2dx\)
 \\ &\ge\( \sqrt{\textstyle\int_{\mathbb{R}^N}|\nabla {\bm u}_P|^2dx} \sqrt{\textstyle\int_{\mathbb{R}^N}|{\bm u}_P|^2|{\bm x}|^2dx}+\sqrt{\textstyle\int_{\mathbb{R}^N}|\nabla {\bm u}_T|^2|{\bm x}|^2dx}\sqrt{\textstyle\int_{\mathbb{R}^N}|{\bm u}_T|^2|{\bm x}|^2dx}\)^2
 \\&\ge \(\sqrt{C_{P,N}}\int_{\mathbb{R}^N}|{\bm u}_P|^2dx+\sqrt{C_{T,N}}\int_{\mathbb{R}^N}|{\bm u}_T|^2dx\)^2
 \\&\ge\(\min\left\{\sqrt{C_{P,N}},\sqrt{C_{T,N}}\right\}\(\int_{\mathbb{R}^N}|{\bm u}_P|^2dx+\int_{\mathbb{R}^N}|{\bm u}_T|^2dx\)\)^2
 \\&=\min\left\{C_{P,N},C_{T,N}\right\}\(\int_{\mathbb{R}^N}|{\bm u}|^2dx\)^2,
\end{aligned}
\end{equation}
where the inequality in the third line follows by applying the Cauchy-Schwarz inequality of the form $(a^2+b^2)(c^2+d^2)\ge (ac+bd)^2$, and where the notation 
 \[\begin{split}
&   C_{P,N}:=\inf_{{\bm u}\in\mathcal{P}}\frac{\int_{\mathbb{R}^N}|\nabla {\bm u}|^2dx\int_{\mathbb{R}^N}|{\bm u}|^2|{\bm x}|^2dx}{\int_{\mathbb{R}^N}|{\bm u}|^2dx}
\\
\text{resp}.\quad & C_{T,N}:=\inf_{{\bm u}\in\mathcal{T}}\frac{\int_{\mathbb{R}^N}|\nabla {\bm u}|^2dx\int_{\mathbb{R}^N}|{\bm u}|^2|{\bm x}|^2dx}{\int_{\mathbb{R}^N}|{\bm u}|^2dx}
\end{split}
\]  
denotes the best HUP constant for poloidal resp. toroidal fields in $\mathcal{D}(\mathbb{R}^N)\setminus\{{\bm 0}\}$ which we abbreviate as ${\bm u}\in \mathcal{P}$ resp. ${\bm u}\in\mathcal{T}$ under the infimum sign.  
  Then we find that the best HUP constant  for solenoidal fields satisfies
\begin{equation}
C_{N}= \inf_{\substack{
{\rm div}{\bm u}=0}}\frac{\int_{\mathbb{R}^N}|\nabla {\bm u}|^2dx\int_{\mathbb{R}^N}|{\bm u}|^2|{\bm x}|^2dx}{\int_{\mathbb{R}^N}|{\bm u}|^2dx}
\ge\min\left\{C_{P,N},C_{T,N}\right\}.
\end{equation}
Combining this inequality with the reverse $C_{N}\le\min\left\{C_{P,N},C_{T,N}\right\}$, we get
\begin{equation}
 C_{N}=\min\left\{C_{P,N},C_{T,N}\right\}.
\label{C_N}
\end{equation}
In this way,  the problem of finding the value of $C_{N}$ is separated into that of $C_{P,N}$ and $C_{T,N}$. %

\section{The case \mbox{\boldmath $u=u_T$}: %
evaluation of $C_{T,N,\gamma}$
}\label{sec:tor}
In this section, we always assume that  ${\bm u}=(u_1,\cdots,u_N)\in\mathcal{D}(\mathbb{R}^N)$ is toroidal. As mentioned in Proposition \ref{prop:prePT}, this field is of zero spherical mean:
\[
 \int_{\mathbb{S}^{N-1}}u_k({r}{\bm \sigma})\mathrm{d}\sigma=0\quad\ \forall{r}>0,\quad \forall k\in\{1,\cdots,N\}.
\]
Then an application of Cauchy-Schwarz inequality yields
\begin{align}
\int_{\mathbb{R}^N}&|\nabla {\bm u}|^2dx \int_{\mathbb{R}^N}|{\bm u}|^2|{\bm x}|^2dx
 \\&=\(\sum_{k=1}^N \int_{\mathbb{R}^N}|\nabla u_k|^2dx\) \(\sum_{k=1}^N \int_{\mathbb{R}^N}(u_k)^2|{\bm x}|^2dx\)
 \\&\ge\(\sum_{k=1}^N\sqrt{\int_{\mathbb{R}^N}|\nabla u_k|^2dx}\sqrt{\int_{\mathbb{R}^N}(u_k)^2|{\bm x}|^2dx}\)^2
 \\&\ge \(\sum_{k=1}^N \mfrac{N+2}{2}\int_{\mathbb{R}^N}(u_k)^2dx\)^2
 =\(\mfrac{N+2}{2}\)^2\(\int_{\mathbb{R}^N}|{\bm u}|^2dx\)^2,\qquad 
\label{HUP_tor}
\end{align}
where the second inequality holds by using the following fact:
\begin{lemma}
\label{lemma:HUP_1}
 Let $f\in C^\infty(\mathbb{R}^N)$ be a scalar field with $\|f\|_{\mathcal{D}}<\infty$ and \[\int_{\mathbb{S}^{N-1}}f({r}{\bm \sigma})\mathrm{d}\sigma=0\quad\text{ for all }{r}>0.\] 
We assume that $\int_{\mathbb{S}^{N-1}}\(f(R {\bm \sigma})\)^2\mathrm{d}\sigma=o(R^{-N})$ as $R\to\infty$. Then it holds that
\[
 \int_{\mathbb{R}^N}|\nabla f|^2dx
 \int_{\mathbb{R}^N}f^2|{\bm x}|^2dx\ge\(\mfrac{N+2}{2}\)^2\(\int_{\mathbb{R}^N}f^2dx\)^2.
\]
Here the equality holds when $f$ is of the form $f({\bm x})=\sum_{k=1}^N B_kx_k e^{-B_0|{\bm x}|^2}$, with $B_0>0$ and  $\{B_1\cdots,B_N\}\subset\mathbb{R}$ being constant numbers. 
\end{lemma}
\begin{proof}
First, let $\nu\in\mathbb{N}$ and let us abbreviate as $f=f({r}{\bm \sigma})$ for ${r}>0$ and ${\bm \sigma}\in\mathbb{S}^{N-1}$. Then  a direct calculation for every $R>0$ yields
\[
\begin{aligned}
\int_0^R &\(\partial_{r}({r}^{-\nu}f)\)^2{r}^{2\nu+N-1}d{r}=\int_0^R\(\partial_{r}f-\nu {r}^{-1}f\)^2{r}^{N-1}d{r}
\\&=\int_0^R\Big(\((\partial_{r}f)^2+\nu^2{r}^{-2}f^2\){r}^{N-1}-2\nu{r}^{N-2}f\partial_{r}f\Big)d{r}
 \\&=\int_0^R\Big(\((\partial_{r}f)^2+\nu^2{r}^{-2}f^2\){r}^{N-1}+\nu(N-2){r}^{N-3}f^2\Big)d{r}-\nu\left[ {r}^{N-2}f^2\right]_{{r}=0}^{{r}=R}
 \\&=\int_0^R\((\partial_{r}f)^2+\alpha_\nu{r}^{-2}f^2\){r}^{N-1}d{r}-\nu R^{N-2}(f(R {\bm \sigma}))^2,
\end{aligned}
\]
where the third equality follows by integration by parts. Hence, by letting $R\to\infty$ after integration of both sides over $\mathbb{S}^{N-1}$, we get
\[
 \int_{\mathbb{R}^N}\(\partial_{r}\({r}^{-\nu}f\)\)^2{r}^{2\nu}dx= \int_{\mathbb{R}^N}\((\partial_{r}f)^2+\alpha_\nu {r}^{-2}f^2\)dx.
\]
Second, let 
\[f=\sum_{\nu\in\mathbb{N}}f_\nu,\qquad  f_\nu\in\mathcal{E}_\nu(\dot{\mathbb{R}}^N)\quad(\forall\nu\in\mathbb{N})\] be the spherical harmonics decomposition of $f$. %
Then we have
 \[
 \begin{aligned}
 \int_{\mathbb{R}^N}|\nabla f|^2dx&
=\sum_{\nu\in\mathbb{N}}\int_{\mathbb{R}^N}|\nabla f_\nu|^2dx
 \\&=\sum_{\nu\in\mathbb{N}}\int_{\mathbb{R}^N}\((\partial_{r}f_\nu)^2+\alpha_\nu{r}^{-2}(f_\nu)^2\)dx
 \\&=\sum_{\nu\in\mathbb{N}}\int_{\mathbb{R}^N}\(\partial_{r}\({r}^{-\nu}f_\nu\)\)^2{r}^{2\nu}dx,%
 \end{aligned}
 \]
 where the last equality follows from the formula derived in the first step. %
Hence, an application of Cauchy-Schwarz inequality yields
 \begin{align}
  \int_{\mathbb{R}^N}&|\nabla f|^2dx \int_{\mathbb{R}^N}|f|^2|{\bm x}|^2dx
  \\&= \(\sum_{\nu\in\mathbb{N}}\int_{\mathbb{R}^N}\(\partial_{r}\({r}^{-\nu}f_\nu\)\)^2|{\bm x}|^{2\nu}dx\)\sum_{\nu\in\mathbb{N}}\int_{\mathbb{R}^N}({r}^{-\nu}f_\nu)^2|{\bm x}|^{2\nu+2}dx
  \\&\ge\(\sum_{\nu\in{\mathbb{N}}}\sqrt{\int_{\mathbb{R}^N}\(\partial_{r}\({r}^{-\nu}f_\nu\)\)^2|{\bm x}|^{2\nu}dx}
  \sqrt{\int_{\mathbb{R}^N}({r}^{-\nu}f_\nu)^2|{\bm x}|^{2\nu+2}dx}\)^2
  \label{CS_tor}
  \\&\ge\(\sum_{\nu\in\mathbb{N}}\int_{\mathbb{R}^N}\(\partial_{r}({r}^{-\nu}f_\nu)\)({r}^{-\nu}f_\nu)|{\bm x}|^{2\nu+1}dx\)^2
  \\&= 
  \(\sum_{\nu\in\mathbb{N}}\mfrac{2\nu+N}{2}\int_{\mathbb{R}^N}\({r}^{-\nu}f_\nu\)^2{r}^{2\nu}dx-\mfrac{1}{2}\lim_{R\to\infty}R^N\int_{\mathbb{S}^{N-1}}\(f(R {\bm \sigma})\)^2\mathrm{d}\sigma\)^2
  \\&= 
  \(\sum_{\nu\in\mathbb{N}}\mfrac{2\nu+N}{2}\int_{\mathbb{R}^N}\({r}^{-\nu}f_\nu\)^2{r}^{2\nu}dx\)^2
  \\&\ge\(\mfrac{N+2}{2}\)^2\(\sum_{\nu\in\mathbb{N}}\int_{\mathbb{R}^N}(f_\nu)^2dx\)^2=\(\mfrac{N+2}{2}\)^2\(\int_{\mathbb{R}^N}f^2dx\)^2,
 \end{align}
 where the equality in the third last line follows by integration by parts with respect to the  measure  $d{r}$, and where the equality in the second last line follows by using the assumption of the lemma. 

Finally, we specify the condition for the equality sign. In order that all the equalities in the above inequalities are simultaneously realized by $f\not\equiv 0$, it must hold that $f_\nu\equiv 0\ (\forall\nu\ge2)$ and that the two integrands in \eqref{CS_tor} are proportional, which leads to
\[
f=f_1\quad\text{ and }\quad \(\partial_{r}({r}^{-1}f_1)\)^2=B^2({r}^{-1}f_1)^2{r}^{2}
\]
for some constant $B\ne 0$, whence
\[
f=f_1\quad\text{ and }\quad  
\partial_{r}({r}^{-1}f_1)=-B{r}\({r}^{-1}f_1\).
\]
 Since ${r}^{-1}f_1({r})$ is convergent  as %
${r}\to\infty$, the solution to this equation must be of the form
\[
f({r}{\bm \sigma})=Y({\bm \sigma}){r}e^{-B{r}^{2}}\quad\text{with}\quad  B>0 ,%
\]
where $Y\in C^\infty(\mathbb{S}^{N-1})$ satisfies the eigenequation $-\triangle_\sigma Y=\alpha_1 Y$. Since the eigenspace of $-\triangle_\sigma$ associated with $\alpha_1$ is spanned by $\{\sigma_1,\cdots,\sigma_N\}$, we consequently have
\[
 f({r}{\bm \sigma})=\sum_{k=1}^NB_k\sigma_k {r}e^{-B{r}^2},\quad\text{that is,}\quad 
f({\bm x})=\sum_{k=1}^N B_kx_k e^{-B|{\bm x}|^2},
\]
as desired. Moreover, a direct computation of the integrals for this case yields
 \begin{equation}
  \left. \begin{aligned}
	  &\int_{\mathbb{R}^N}|\nabla f|^2dx=\frac{N+2}{4}\(\frac{\pi}{2B}\)^{N/2}\sum_{k=1}^N B_k^2,
	  \\& \int_{\mathbb{R}^N}f^2|{\bm x}|^2dx= \frac{N+2}{4(2B)^{2}}\(\frac{\pi}{2B}\)^{N/2}\sum_{k=1}^NB_k^2,
	  \\&
	  \int_{\mathbb{R}^N}f^2dx=\frac{1}{4B}\(\frac{\pi}{2B}\)^{N/2}\sum_{k=1}^N B_k^2,
	 \end{aligned}
  \right\}
\label{BBB}
 \end{equation}
 which clearly achieves the equality in the inequality of the lemma.
\end{proof}
In order that both the equalities of the two inequalities in \eqref{HUP_tor} are simultaneously realized by ${\bm u}\not\equiv {\bm 0}$, it must hold that there exists $C>0$ such that
\begin{align}
 &\int_{\mathbb{R}^N}|\nabla  u_k|^2dx=C \int_{\mathbb{R}^N}(u_k)^2|{\bm x}|^2dx
\\
\text{and}\quad
& \int_{\mathbb{R}^N}|\nabla u_k|^2dx
 \int_{\mathbb{R}^N}(u_k)^2|{\bm x}|^2dx=\(\mfrac{N+2}{2}\)^2\(\int_{\mathbb{R}^N}(u_k)^2dx\)^2
\end{align}
for all $k\in\{1,\cdots,N\}$. Notice here that the second equation requires ${\bm u}$ to be 
\[
\begin{aligned}
 u_k({\bm x})=\sum_{j=1}^N B_{k,j} x_j e^{-B_{k,0}|{\bm x}|^2}
\end{aligned}
\]
with constant numbers $\{B_{k,0}>0\}_{k\in\{1,\cdots,N\}}$ and $\{B_{k,j}\}_{k,j\in\{1,\cdots,N\}}\subset\mathbb{R}$; in view of this together with the first equation, we see from \eqref{BBB} that 
\[
 \begin{aligned}
  \frac{\int_{\mathbb{R}^N}|\nabla u_k|^2dx}{\int_{\mathbb{R}^N}(u_k)^2dx}
=\frac{\frac{N+2}{4}\(\frac{\pi}{2B_{k,0}}\)^{N/2}\sum_{j=1}^N B_{k,j}^2}{ \frac{N+2}{(2B_{k,0})^{2}}\(\frac{\pi}{2B_{k,0}}\)^{N/2}\sum_{j=1}^NB_{k,j}^2}=B_{k,0}^2
\end{aligned}
\]
must be independent of $k$, which is equivalent to $B_{1,0}=B_{2,0}=\cdots=B_{N,0}$.
Therefore, ${\bm u}$ must be of the form
\[
 {\bm u}({\bm x})=e^{-c|{\bm x}|^2}{\bm u}_0({\bm x}), \quad {\bm u}_0({\bm x})=\({\bm c}_1\cdot {\bm x},\ {\bm c}_2\cdot {\bm x},\ \cdots,\ {\bm c}_N\cdot {\bm x}\),
\]
where $c>0$ and where $\left\{{\bm c}_1,{\bm c}_2\cdots,{\bm c}_N\right\}\subset\mathbb{R}^N$ are $N$ constant vectors. Since ${\bm u}$ is toroidal if and only if so is ${\bm u}_0$, it must hold that the constant matrix $({\bm c}_1,\cdots,{\bm c}_N)$ is trace-free and satisfy
\[
\begin{aligned}
 {\bm x}\cdot {\bm u}_0({\bm x})
&= x_1({\bm c}_1\cdot {\bm x})+\cdots +x_N ({\bm c}_N\cdot {\bm x})
 =\sum_{j,k=1}^N c_{jk}x_j x_k
 \\&=\sum_{k=1}^N c_{kk}^2 x_k^2+\sum_{j<k}(c_{jk}+c_{kj})x_jx_k=0\qquad \forall{\bm x}\in\mathbb{R}^N,
\end{aligned}
\]
whence $c_{jk}=-c_{kj}$ $(\forall j,k)$, that is,  $({\bm c}_1,\cdots,{\bm c}_N)$ is antisymmetric. In summary, we have obtained the following result:
\begin{theorem}
\label{theorem:tor}
 Let ${\bm u}\in\mathcal{D}(\mathbb{R}^N)$ be a toroidal field. Then the inequality
\[
 \int_{\mathbb{R}^N}|\nabla {\bm u}|^2dx \int_{\mathbb{R}^N}|{\bm u}|^2|{\bm x}|^2dx\ge C_{T,N}\(\int_{\mathbb{R}^N}|{\bm u}|^2dx\)^2
\]
 holds with the best constant $C_{T,N}=\frac{(N+2)^2}{4}$. Moreover, $C_{T,N}$ is attained by ${\bm u}={\bm u}_0$ with the profile
 \begin{equation}
  {\bm u}_0({\bm x})=\({\bm c}_1\cdot {\bm x},\ \cdots \ , {\bm c}_N\cdot {\bm x}\)e^{-c|{\bm x}|^2},
 \end{equation}
where $c$ is any positive constant number  and $\{{\bm c}_1,\cdots, {\bm c}_N\}\subset\mathbb{R}^N$ are any $N$ constant vectors with the matrix $({\bm c}_1,\cdots,{\bm c}_N)$ being antisymmetric. 
Such ${\bm u}_0$ is unique up to that profile, under the additional condition $\lim_{R\to\infty}R^N\int_{\mathbb{S}^{N-1}}|{\bm u}(R {\bm \sigma})|^2\mathrm{d}\sigma=0$.
\end{theorem}
In view of this result and \eqref{CN_CFL}, 
we find that the toroidal condition gives the same best HUP constant as the curl-free condition.
\section{The case \mbox{\boldmath $u=u_{P}$}: evaluation of $C_{P,N}$}
 \label{sec:pol}
 Throughout this section, ${\bm u}$ is a smooth poloidal field with compact support on $\dot{\mathbb{R}}^N$ expressed as
\[{\bm u}%
={\bm D}f,%
\]
where $f\in C_c^\infty(\dot{\mathbb{R}}^N)\setminus\{{\bm 0}\}$ is of zero spherical mean.%
\subsection{Spherical harmonics decomposition of poloidal fields} 

In the same way as in the proof of Lemma~\ref{lemma:HUP_1}, we again express the spherical harmonics expansion of $f$ as
\[
 f=\sum_{\nu\in\mathbb{N}}f_\nu,\ \quad f_\nu\in\mathcal{E}_\nu(\dot{\mathbb{R}}^N)\qquad(\forall\nu\in\mathbb{N}),
\]
and we write as
\[
 {\bm u}_\nu={\bm D}f_\nu %
\qquad(\forall\nu\in\mathbb{N})
\]
which clearly belongs to $\mathcal{P}_\nu(\dot{\mathbb{R}}^N)$. 
Then %
an application of Cauchy Schwarz inequality by way of the $L^2(\mathbb{S}^{N-1})$-orthogonality formulae \eqref{SH_orth} yields that%
\[
 \begin{aligned}
  \int_{\mathbb{R}^N}|\nabla {\bm u}|^2dx
  \int_{\mathbb{R}^N}&|{\bm u}|^2|{\bm x}|^2dx=\(\sum_{\nu\in\mathbb{N}}\int_{\mathbb{R}^N}|\nabla {\bm u}_\nu|^2dx\)\(\sum_{\nu\in\mathbb{N}}\int_{\mathbb{R}^N}|{\bm u}_\nu|^2|{\bm x}|^2dx\)
\\&\ge \sum_{\nu\in\mathbb{N}}\int_{\mathbb{R}^N}|\nabla {\bm u}_\nu|^2dx \int_{\mathbb{R}^N}|{\bm u}_\nu|^2|{\bm x}|^2dx
\\&\ge\sum_{\nu\in\mathbb{N}}
C_{P,N,\nu}\int_{\mathbb{R}^N}|{\bm u}_\nu|^2dx
\ge\inf_{\nu\in\mathbb{N}}C_{P,N,\nu}\int_{\mathbb{R}^N}|{\bm u}|^2dx.
 \end{aligned}
\]
Here, for every $\nu\in\mathbb{N}$, the notation%
\begin{equation}
 C_{P,N,\nu}=\inf_{{\bm u}\in {\bm D}\mathcal{E}_\nu}\frac{\int_{\mathbb{R}^N}|\nabla {\bm u}|^2dx\int_{\mathbb{R}^N}|{\bm u}|^2|{\bm x}|^2dx}{\int_{\mathbb{R}^N}|{\bm u}|^2|{\bm x}|^2dx}%
\label{CPNn}
\end{equation}
denotes the best $\rm HUP$ constant for the poloidal fields (in $C_c^\infty(\dot{\mathbb{R}}^N)^N$) generated by $\mathcal{E}_\nu(\dot{\mathbb{R}}^N)$, %
where the abbreviation ${\bm u}\in {\bm D}\mathcal{E}_\nu$ under the infimum sign means that 
\[{\bm u}\in%
\left\{{\bm D}f\ ;\ f\in \mathcal{E}_\nu(\dot{\mathbb{R}}^N)\cap C_c^\infty(\dot{\mathbb{R}}^N)\setminus\{0\}\right\}.%
\] 
Then we find the inequality $C_{P,N}\ge \inf_{\nu\in\mathbb{N}}C_{P,N,\nu}$ from the definition of $C_{P,N}$. Since the reverse $C_{P,N}\le \inf_{\nu\in\mathbb{N}}C_{P,N,\nu}$ also holds true, it turns out that 
\begin{equation}
 C_{P,N}=\inf_{\nu\in\mathbb{N}}C_{P,N,\nu}.
\label{C_PN}
\end{equation}
Therefore, the problem of computing $C_{P,N}$ is reduced to that of $C_{P,N,\nu}$. 
\subsection{Evaluation of $C_{P,N,\nu}$}
\label{subsec:C_PNn}
Here we fix $\nu\in\mathbb{N}$ and assume that ${\bm u}$ is of the form
\begin{equation}
 {\bm u}={\bm D}f,\quad f\in \mathcal{E}_\nu(\dot{\mathbb{R}}^N)\cap C_c^\infty(\dot{\mathbb{R}}^N).
\end{equation}
For simplicity, we use the  abbreviations
\begin{equation}
 \alpha=\alpha_\nu, \quad \partial={r}\partial_{r}\quad\text{ and }\quad \partial'={r}\partial_{r}'=\partial+N-1.
\end{equation}
In order to evaluate $C_{P,N,\nu}$, we have to compute the integrals in the HUP inequality. 
To this end, a direct calculation yields
\begin{align}
 {\bm u}&=\({\bm \sigma}\triangle_\sigma -{r}\partial_{r}'\nabla_{\!\sigma}\)f=-{\bm \sigma}\alpha f-\nabla_{\!\sigma}\partial' f,
 \\
 \partial {\bm u}&=-{\bm \sigma}\alpha\partial f-\partial' \partial\nabla_{\!\sigma}f,
 \\
 - \triangle_\sigma {\bm u}&=\alpha\:\!\triangle_\sigma ({\bm \sigma}f)+\partial'\triangle_\sigma\nabla_{\!\sigma}f
 \\&=\alpha \({\bm \sigma}\(\triangle_\sigma f-(N-1)f\)+2\nabla_{\!\sigma}f\)
 \\&\quad +\partial'\(\nabla_{\!\sigma}\triangle_\sigma f-2 {\bm \sigma}\triangle_\sigma f +(N-3)\nabla_{\!\sigma}f\)
 \\&=\alpha\(-{\bm \sigma}\(\alpha+N-1\)f+2\nabla_{\!\sigma}f\)
 \\*&\quad +\partial'\(-\alpha\nabla_{\!\sigma}f+2 {\bm \sigma}\alpha f+(N-3)\nabla_{\!\sigma}f\)
 \\&=\alpha {\bm \sigma}\(2\:\! \partial'f-(\alpha+N-1)f\)+\nabla_{\!\sigma}\(2\:\!\alpha f+\(-\alpha+N-3\)\partial' f\)
 \\&=\alpha{\bm \sigma}\(2\partial f-(\alpha-N+1)f\)+\nabla_{\!\sigma}\(2\:\!\alpha f+(-\alpha+N-3)\partial'f\),
\end{align}
where the equality in the fourth line follows by using \eqref{comm}. Taking the absolute square or scalar product of them yields
\begin{align}
 |{\bm u}|^2&=\alpha^2 f^2+|\nabla_{\!\sigma}\partial'f|^2,%
\\
|\partial {\bm u}|^2&=\alpha^2(\partial f)^2+|\nabla_{\!\sigma}\partial'\partial f|^2,
\\
 -{\bm u}\cdot\triangle_\sigma {\bm u}
&= -\alpha^2 f\big(2\partial f-(\alpha-N+1)f\big)
 \\
 &\quad -\nabla_{\!\sigma}(\partial'f)\cdot\nabla_{\!\sigma}\big(2\alpha f+(-\alpha+N-3)\partial'f\big).
\end{align}
Integration by parts of them over $\mathbb{S}^{N-1}$ (for any ${r}$) then yields the following calculations:
\begin{align}
\frac{1}{\alpha} &\int_{\mathbb{S}^{N-1}} |{\bm u}|^2\mathrm{d}\sigma%
  = \int_{\mathbb{S}^{N-1}}\(\alpha f^2+\(\partial'f\)^2\)\mathrm{d}\sigma,
\\&= \int_{\mathbb{S}^{N-1}}\(\alpha f^2+(\partial f+(N-1)f)^2\)\mathrm{d}\sigma
 \\&= \int_{\mathbb{S}^{N-1}}\Big(\(\alpha+(N-1)^2\)f^2+(\partial f)^2+2(N-1) f \partial f\Big)\mathrm{d}\sigma,
\label{L2S_u}
\\
\frac{1}{\alpha}& \int_{\mathbb{S}^{N-1}}|\partial {\bm u}|^2\mathrm{d}\sigma
 \\&=\int_{\mathbb{S}^{N-1}}\Big(\(\alpha+(N-1)^2\)(\partial f)^2+(\partial^2f)^2+2(N-1)(\partial f)\partial^2f\Big)\mathrm{d}\sigma\qquad 
\label{L2S_du}
\intertext{and}  
\frac{1}{\alpha}&  \int_{\mathbb{S}^{N-1}}|\nabla_{\!\sigma}{\bm u}|^2\mathrm{d}\sigma=\frac{1}{\alpha}\int_{\mathbb{S}^{N-1}}\(
  \begin{aligned}
   &-\alpha^2 f\big(2\partial f-(\alpha-N+1)f\big)
   \\
   & +(\triangle_\sigma \partial'f)\(2\alpha f+(-\alpha+N-3)\partial'f\)
  \end{aligned}
  \)\mathrm{d}\sigma
  \\&=\int_{\mathbb{S}^{N-1}}
  \(\begin{aligned}
     &  -2\alpha f\partial f+\alpha(\alpha-N+1)f^2
     \\
     & + (\partial'f)\Big(-2\alpha f+(\alpha-N+3)\partial'f\Big)
    \end{aligned}\)
  \mathrm{d}\sigma
  \\&= \mathop{\int}_{\mathbb{S}^{N-1}}\(
  \begin{aligned}
   &(\alpha-N+3)(\partial f)^2			\\&
   +\(\begin{aligned}
       &   -2\alpha+\big(-2\alpha+(\alpha-N+3)(N-1)\big)
      \\& +(N-1)\(\alpha-N+3\)\end{aligned}
   \)f\partial f
  \\&+\Big(\alpha(\alpha-N+1)+(N-1)\big(-2\alpha+(\alpha-N+3)(N-1)\big)\Big)f^2				      \end{aligned}
  \)\mathrm{d}\sigma
  \\&= \int_{\mathbb{S}^{N-1}}\(
  \begin{aligned}
   &(\alpha-N+3)(\partial f)^2+2(N-3)(\alpha-N+1)f\partial f
   \\&+(\alpha-N+1)\big(\alpha+(N-1)(N-3)\big)f^2
  \end{aligned}\)
  \mathrm{d}\sigma.
\label{L2S_Su}
 \end{align}
Now, along the idea of \cite{CFL_HUP,Tertikas-Z}, let us introduce the transformation
\begin{equation}
f={r}^{\lambda} g
\end{equation}
with $\lambda\in\mathbb{R}$ determined later, 
and notice for any $\beta\in\mathbb{R}$ and ${\bm \sigma}\in\mathbb{S}^{N-1}$ the integration by parts formulae
\begin{align}
 \int_{\mathbb{R}_+}&(f\partial f){r}^{\beta-1}d{r}=-\frac{\beta}{2}\int_{\mathbb{R}_+}f^2 {r}^{\beta-1} d{r}=-\frac{\beta}{2}\int_{\mathbb{R}_+}g^2{r}^{2\lambda+\beta-1}d{r},
\label{fdf}
 \\
 \int_{\mathbb{R}_+}&(\partial f)^2{r}^{\beta-1} dr
  =\int_{\mathbb{R}_+}\(\lambda {g}+\partial {g}\)^2r^{\beta+2\lambda-1} dr
 \\&= \int_{\mathbb{R}_+}\(\lambda^2{g}^2+(\partial {g})^2+2\lambda{g}\partial {g}\){r}^{\beta+2\lambda}dt
 \\&=\int_{\mathbb{R}_+}\((\partial {g})^2-\lambda\(\lambda+\beta\){g}^2\)r^{\beta+2\lambda-1}dr,
\label{df2}
\\  \int_{\mathbb{R}_+}&(\partial f)(\partial^2 f){r}^{\beta-1}dr=-\frac{\beta}{2}\int_{\mathbb{R}_+}(\partial f)^2{r}^{\beta-1} d{r}
\\&=-\frac{\beta}{2}\int_{\mathbb{R}_+}\((\partial g)^2-\lambda(\lambda+\beta)g^2\){r}^{\beta+2\lambda-1}d{r},
\label{dfddf}
\\
\int_{\mathbb{R}_+}&\({g}\partial^2 {g}\)\,{r}^{\beta-1}dr=\int_{\mathbb{R}}\(\partial \({g}\partial {g}\)-(\partial {g})^2\){r}^{\beta-1} d{r}
 \\&=\int_{\mathbb{R}}\(-\beta {g}\partial {g}-(\partial {g})^2\){r}^{\beta-1} d{r}=\int_{\mathbb{R}_+}\(\frac{\beta^2}{2}{g}^2-(\partial {g})^2\){r}^{\beta-1} d{r}
\label{gddg}
\\
 \int_{\mathbb{R}_+}&  (\partial^2 f)^2{r}^{\beta-1} d{r}
=\int_{\mathbb{R}_+}\(\lambda^2{g}+2\lambda\partial {g}+\partial^2 {g}\)^2{r}^{2\lambda+{\beta-1}} d{r}
  \\&=\int_{\mathbb{R}_+}
\(\begin{aligned}
&\lambda^4{g}^2+4\lambda^2(\partial {g})^2+(\partial^2 {g})^2+4\lambda^3{g}\partial {g}
\\&+2\lambda^2{g}\partial^2 {g}+4\lambda(\partial {g})\partial^2 {g}\end{aligned}
\){r}^{2\lambda+\beta-1}d{r}
  \\&=\int_{\mathbb{R}_+}\(
  \begin{aligned}
   & \lambda^4{g}^2+4\lambda^2(\partial {g})^2+(\partial^2 {g})^2-2\lambda^3(2\lambda+\beta){g}^2
   \\&+2\lambda^2\(\tfrac{(2\lambda+\beta)^2}{2}{g}^2-(\partial {g})^2\)
  -2\lambda(2\lambda+\beta)(\partial {g})^2
\end{aligned}
  \){r}^{2\lambda+\beta-1}d{r}
  \\&=\int_{\mathbb{R}_+}\Big(\lambda^2(\lambda+\beta)^2{g}^2-2\lambda\(\lambda+\beta\)(\partial {g})^2+(\partial^2 {g})^2\Big){r}^{2\lambda+\beta-1}d{r},\label{ddf2}
\end{align} 
where the second last equality follows by applying \eqref{gddg} to $2\lambda+\beta$ instead of $\beta$. Then, by way of applying \eqref{df2} and \eqref{fdf} to $\beta+N$ instead of $\beta$, the integration of \eqref{L2S_u}  with respect to the measure ${r}^{\beta+N-1}d{r}$ yields 
\begin{align}
 \frac{1}{\alpha}&\int_{\mathbb{R}^N}|{\bm u}|^2|{\bm x}|^\beta dx=\frac{1}{\alpha}\iint_{\mathbb{R}_+\times\mathbb{S}^{N-1}}|{\bm u}|^2{r}^{\beta+N-1}d{r}\mathrm{d}\sigma
\\&= \iint_{\mathbb{R}_+\times\mathbb{S}^{N-1}}\Big(\(\alpha+(N-1)^2\)f^2+(\partial f)^2+2(N-1) f \partial f\Big){r}^{\beta+N-1}\mathrm{d}\sigma
 \\&=\iint_{\mathbb{R}_+\times\mathbb{S}^{N-1}}\(
\begin{aligned}
& \(\alpha+(N-1)^2\)g^2+(\partial {g})^2
\\&-\lambda\(\lambda+\beta+N\){g}^2-(N-1)(\beta+N)g^2
\end{aligned}
\){r}^{2\lambda+\beta+N-1}d{r}\mathrm{d}\sigma,
\\&=\iint_{\mathbb{R}_+\times\mathbb{S}^{N-1}}\bigg( (\partial g)^2+\Big(\alpha-\alpha_{\lambda+1}-(\lambda+N-1)\beta\Big)g^2\bigg){r}^{2\lambda+\beta+N-1}d{r}\mathrm{d}\sigma.
\end{align}
By applying \eqref{df2},  \eqref{ddf2}, \eqref{dfddf} and \eqref{fdf} to the case $\beta=N-2$, the integration of $\eqref{L2S_du}+\eqref{L2S_Su}$ with respect to the measure ${r}^{N-1}d{r}$ also yields
\begin{align}
\frac{1}{\alpha}& \int_{\mathbb{R}^N}|\nabla {\bm u}|^2dx=\frac{1}{\alpha}\iint_{\mathbb{R}_+\times\mathbb{S}^{N-1}}
\(|\partial {\bm u}|^2+|\nabla_{\!\sigma}{\bm u}|^2\){r}^{N-3}d{r}\mathrm{d}\sigma
\\&= \mathop{\iint}_{\mathbb{R}_+\times\mathbb{S}^{N-1}}\(
  \begin{aligned}
   &\(2\alpha+N^2-3N+4\)(\partial f)^2+(\partial^2f)^2
\\&+2(N-1)(\partial f)\partial^2f+2(N-3)(\alpha-N+1)f\partial f
   \\&+(\alpha-N+1)\big(\alpha+(N-1)(N-3)\big)f^2
  \end{aligned}
\){r}^{N-3}d{r}\mathrm{d}\sigma
\\&=\mathop{\iint}_{\mathbb{R}_+\times\mathbb{S}^{N-1}}\(
\begin{aligned}
 &\(2\alpha+N^2-3N+4\)\((\partial g)^2-\alpha_{\lambda}g^2\)
\\&+\alpha_\lambda^2g^2-2\alpha_{\lambda} (\partial g)^2+(\partial^2 g)^2
\\&-(N-1)(N-2)\((\partial g)^2-\alpha_{\lambda} g^2\)
\\&-(N-3)(\alpha-N+1)(N-2)g^2
\\&+(\alpha-N+1)\(\alpha+(N-1)(N-3)\)g^2
\end{aligned}
\){r}^{2\lambda+N-3}d{r}\mathrm{d}\sigma
\\&=\iint_{\mathbb{R}_+\times\mathbb{S}^{N-1}}\(
\begin{aligned}
&(\partial^2 g)^2+2(\alpha-\alpha_\lambda+1)(\partial g)^2		
\\&+\(\alpha-\alpha_{\lambda-1}\)\(\alpha-\alpha_{\lambda+1}\)
g^2				\end{aligned}
\){r}^{2\lambda+N-3}d{r}\:\!\mathrm{d}\sigma.
\end{align}
Now, we choose
\[
 \lambda=\nu-1
\]
in order that $\alpha_{\lambda+1}=\alpha_\nu\(=\alpha\)$. 
Then the above results turn into
\begin{align}
 \frac{1}{\alpha}\int_{\mathbb{R}^N}|{\bm u}|^2|{\bm x}|^\beta dx
&=\iint_{\mathbb{R}_+\times\mathbb{S}^{N-1}}\Big( (\partial g)^2 -(\nu+N-2)\beta g^2\Big){r}^{2\nu+\beta+N-3}d{r}\mathrm{d}\sigma,
\\
\frac{1}{\alpha}\int_{\mathbb{R}^N}|\nabla {\bm u}|^2dx&=\iint_{\mathbb{R}_+\times\mathbb{S}^{N-1}}\Big(
(\partial^2 g)^2+2(2\nu+N-2)(\partial g)^2		
\Big){r}^{2\nu+N-5}d{r}\:\!\mathrm{d}\sigma.
\end{align}
Furthermore, let us change the radial coordinate from ${r}$ into
\[
 {s}={r}^2
\]
which obeys the chain rules: $ds=2{r}d{r}$ %
and  $\partial%
=2s\partial_s$. %
Abbreviating as $g=g({\bm \sigma}\sqrt{s})$ and $\partial_{s}g=\partial_{s}\(g({\bm \sigma}\sqrt{s})\)$, we have the following calculation:
\[
 \begin{aligned}
 \frac{1}{\alpha_\nu}&\int_{\mathbb{R}^N}|{\bm u}|^2|{\bm x}|^\beta dx
=2\iint_{\mathbb{R}_+\times\mathbb{S}^{N-1}}\Big(
  ({s}\partial_{s} g)^2-\tfrac{\nu+N-2}{4}\beta g^2\Big){s}^{\nu+\frac{\beta+N-4}{2}}d{s}\mathrm{d}\sigma,
\\
\frac{1}{\alpha_\nu}&\int_{\mathbb{R}^N}|\nabla{\bm u}|^2 dx
\\&=\iint_{\mathbb{R}_+\times\mathbb{S}^{N-1}}
\Big(   \((2s\partial_s)^2g\)^2+2(2\nu+N-2)(2s\partial_sg)^2
\Big)
(\sqrt{s})^{2\nu+N-6}\frac{ds}{2}\mathrm{d}\sigma 
\\&=4\iint_{\mathbb{R}_+\times\mathbb{S}^{N-1}}
  \Big(2\big({s}\partial_s(s\partial_s g)\big)^2+(2\nu+N-2)(s\partial_sg)^2
\Big)
  {s}^{\nu+\frac{N}{2}-3}ds\:\!\mathrm{d}\sigma
\\&=4\iint_{\mathbb{R}_+\times\mathbb{S}^{N-1}}
\( 2\({s}\partial_{s}g+s^2\partial_{s}^2g\)^2+(2\nu+N-2)(s\partial_sg)^2
\)
  {s}^{\nu+\frac{N}{2}-3}ds\:\!\mathrm{d}\sigma
\\&=4\mathop{\iint}_{\mathbb{R}_+\times\mathbb{S}^{N-1}}\Big(
2\big({s}^2\partial_{s}^2g\big)^2+4{s}^3(\partial_{s}^2g)\partial_{s}g
+(2\nu+N){s}^2(\partial_{s}g)^2
\Big)
  {s}^{\nu+\frac{N}{2}-3}d{s}\:\!\mathrm{d}\sigma
  \\&=8\iint_{\mathbb{R}_+\times\mathbb{S}^{N-1}}({s}^2\partial_{s}^2g)^2
  {s}^{\nu+\frac{N}{2}-3}d{s}\:\!\mathrm{d}\sigma.
 \end{aligned}
\]
Therefore, we have obtained
 \begin{equation}
\begin{aligned}
 \int_{\mathbb{R}^N}|{\bm u}|^2|{\bm x}|^{2\beta} dx
 &=2\alpha_\nu\int_{\mathbb{S}^{N-1}}P_\beta [g]\,\mathrm{d}\sigma ,
 \\
 \int_{\mathbb{R}^N}|\nabla {\bm u}|^2 dx
 &=\red{8\alpha_\nu\int_{\mathbb{S}^{N-1}}Q[g]\,\mathrm{d}\sigma}
\end{aligned} 
\end{equation}
for any $\beta\in\mathbb{R}$, where we put
\begin{equation}
 \left\{\begin{aligned}
 P_{\beta}[g]&:=\int_{\mathbb{R}_+}\Big(
  ({s}\partial_{s} g)^2-\tfrac{\nu+N-2}{2}\beta g^2\Big){s}^{\nu+\frac{N}{2}+\beta-2}d{s},
\\
Q[g]&:=\red{\int_{\mathbb{R}_+}(\partial_{s}^2g)^2
  {s}^{\nu+\frac{N}{2}+1}d{s}},
\end{aligned}
\right.
\end{equation}
which we view as functionals on the set of the one-dimensional functions \[\left\{{s}\mapsto g({\bm \sigma}\sqrt{s})\right\}_{{\bm \sigma}\in\mathbb{S}^{N-1}}\subset C_c^\infty(\mathbb{R}_+)\]
parameterized by ${\bm \sigma}$.
By an application of Cauchy-Schwarz inequality, we then get
\[
\begin{aligned}
&\frac{\int_{\mathbb{R}^N}|\nabla {\bm u}|^2dx\int_{\mathbb{R}^N}|{\bm u}|^2|{\bm x}|^2dx}{\red{4\(\int_{\mathbb{R}^N}|{\bm u}|^2dx\)^2}}
 =\frac{\int_{\mathbb{S}^{N-1}}Q[g]\mathrm{d}\sigma \int_{\mathbb{S}^{N-1}}P_1[g]\mathrm{d}\sigma}{\(\int_{\mathbb{S}^{N-1}}P_0[g]\mathrm{d}\sigma\)^2}
 \\&\quad \ge \(\frac{\int_{\mathbb{S}^{N-1}}\sqrt{Q[g]P_1[g]}\,\mathrm{d}\sigma}{\int_{\mathbb{S}^{N-1}}P_0[g]\mathrm{d}\sigma}\)^2
\ge\inf_{g\in C_c^\infty(\mathbb{R}_+)}\(\frac{\sqrt{Q[g]P_1[g]}}{P_0[g]}\)^2,
\end{aligned}
\]
where and hereafter we abbreviate as $\displaystyle\inf_{g\in C_c^\infty(\mathbb{R}_+)}$ instead of $\displaystyle\inf_{g\in C_c^\infty(\mathbb{R}_+)\setminus\{0\}}$. 
Taking the infimum of both sides over ${\bm u}\in {\bm D}\mathcal{E}_\nu$ then yields from \eqref{CPNn} the inequality
\[
C_{P,N,\nu}\ge \red{4}\inf_{g\in C_c^\infty(\mathbb{R}_+)}R[g],\quad\text{ where we put}\quad R[g]=\frac{Q[g]P_1[g]}{\(P_0[g]\)^2}.
\]
On the other hand, for any $g\in C_c^\infty(\mathbb{R}_+)\setminus\{0\}$ we set 
\begin{equation}
{\bm u}({\bm x})={\bm D}\Big(|{\bm x}|^{\nu-1}g(|{\bm x}|^2)Y({\bm x}/|{\bm x}|)\Big)%
\label{f=gY}
\end{equation}
in order that ${\bm u}\in {\bm D}\mathcal{E}_\nu$, 
where $Y\in C^\infty(\mathbb{S}^{N-1})$ is any one of the spherical harmonic functions of degree $\nu$. 
Then, after applying to \eqref{f=gY} the same calculations as above, we see that
 \begin{equation}
 \mathop{\int}_{\mathbb{R}^N}|{\bm u}|^2|{\bm x}|^{2\beta} dx
 =2\alpha_\nu P_\beta [g]\mathop{\int}_{\mathbb{S}^{N-1}}Y\mathrm{d}\sigma 
 \quad\text{ and }\quad 
 \mathop{\int}_{\mathbb{R}^N}|\nabla {\bm u}|^2 dx
 =\red{8\alpha_\nu Q[g]}\mathop{\int}_{\mathbb{S}^{N-1}}Y\mathrm{d}\sigma,
\end{equation}
whence we see that 
\begin{equation}
\begin{gathered}
c\int_{\mathbb{R}^N}|{\bm u}|^2|{\bm x}|^{2\beta}dx=P_\beta [g]
\quad\text{ and }\quad c\int_{\mathbb{R}^N}|\nabla {\bm u}|^2dx=\red{4\;\!Q[g]}
\end{gathered}
\label{Df_Rg}%
 \end{equation}
hold with some constant $c>0$ depending only on $N$ and $\nu$; in particular, we get
\[
  C_{P,N,\nu}\le \frac{\int_{\mathbb{R}^N}|\nabla {\bm u}|^2dx\int_{\mathbb{R}^N}|{\bm u}|^2|{\bm x}|^2dx}{\(\int_{\mathbb{R}^N}|{\bm u}|^2dx\)^2}=\red{4\:\!R[g]}.
\]
Taking the infimum of both sides over $g\in C_c^\infty(\mathbb{R}_+)\setminus\{0\}$ yields
\[\quad C_{P,N,\nu}\le \red{4}\inf_{g\in C_c^\infty(\mathbb{R}_+)}R[g]\]
as the reverse to the aforementioned inequality, whence we consequently obtain
\begin{equation}
 C_{P,N,\nu}=\red{4}\inf_{g\in C_c^\infty(\mathbb{R}_+)}R[g].
\label{C_PNn}
\end{equation}

In summary, the evaluation of $C_{P,N,\nu}$ is reduced to the one-dimensional minimization problem for the functional
\begin{equation}
 R[g]=\frac{Q[g]P_1[g]}{(P_0[g])^2}%
\qquad\text{for}\quad  {g}={g}({x})\in C_c^\infty(\mathbb{R}_+)\setminus\{0\},
\label{Rg}
\end{equation}
where \red{the factors of $R[g]$ are given by}
\begin{equation}
 \left\{\begin{aligned}
	 P_\beta[g]&=\int_0^\infty\Big(
	{x}^2(\red{g'})^2-\beta{\varepsilon} g^2\Big){x}^{\mu+\beta-2}d{x}\qquad\red{(\beta=0,1)},
	\\
	Q[g]&= \int_0^\infty(\red{g''})^2
	{x}^{\mu+1}d{x},
       \end{aligned}
\ \right.
\label{P0P2Q}
\end{equation}
\red{with the notations $g'=\dfrac{dg}{dx}$, \ $g''=\dfrac{d^2g}{dx^2}$},
 \begin{equation}
	       \mu=\frac{N}{2}+\nu \quad \text{ and }
\quad  {\varepsilon}=\frac{\nu+N-2}{2}.
\label{mu_and_ep}
\end{equation} 
.

 \section{Solution to the minimization problem for $R[g]$}
\label{sec:Laguerre}
\red{The goal of this section is to solve the one-dimensional minimization problem for the functional \eqref{Rg}; however, we treat $\mu$ and $\nu$ as continuous real positive parameters 
including \eqref{mu_and_ep} as special cases.  
Moreover,  we consider the test functions $g$ in a slightly wider space than $C_c^\infty(\mathbb{R}_+)$ which in particular admits those with non-zero values at the origin.}

\red{Let $\mu>0$ and let $C_\mu^\infty(\mathbb{R}_+)$ be the set of $\mathbb{R}$-valued smooth functions ${g}$ on the half line $\mathbb{R}_+=\{x>0\}$ satisfying
\begin{equation}
\begin{cases}
\displaystyle\ 
\limsup_{x\to 0}\(|g(x)|+|g'(x)|\)<\infty,
\vspace{0.5em}\\
 \displaystyle\lim_{x\to\infty}\(x^{\mu+1}({g}'(x))^2+x^\mu \({g}(x)\)^2\)=0
,
\end{cases}
\label{boundary_f}
\end{equation}
where and hereafter we use the notations ${g}'=\dfrac{d{g}}{dx}$ and ${g}''=\dfrac{d^2{g}}{dx^2}$. 
We work on the space $\mathcal{D}_\mu(\mathbb{R}_+)$ defined as the Hilbert space completion of the set
\[
 \left\{{g}\in C_\mu^\infty(\mathbb{R}_+)\ ;\ \|{g}\|_{\mathcal{D}_\mu}<\infty\right\}
\]
with respect to the norm $\|\cdot\|_{\mathcal{D}_\mu}$ given by
\begin{equation}
\|{g}\|_{\mathcal{D}_\mu}:=\(4\int_0^\infty({g}'')^2x^{\mu+1}dx+\int_0^\infty({g}')^2(x+1)x^\mu dx+\int_0^\infty {g}^2 x^{\mu-1}dx\)^{\frac{1}{2}} ,
\label{norm_Dm}
\end{equation}
in order that all the integrals in \eqref{P0P2Q} are well defined for all $g\in\mathcal{D}_\mu(\mathbb{R}_+)$ in the standard sense. 
In the above setting, the answer to the one-dimensional minimization problem 
is given as follows:}
\begin{theorem}
\label{theorem:3}
\red{
Let $\varepsilon>0$ with $ \varepsilon<\mu^2/4$. Then the inequality 
 \begin{equation}
\begin{aligned}
     \int_0^\infty (g'')^2x^{\mu+1}dx\int_0^\infty &\(x^2(g')^2-\varepsilon g^2\)x^{\mu-1} dx
\\&\ge\frac{1}{4}\(\sqrt{\mu^2-4\:\!\varepsilon}+1\)^2 \int_0^\infty \(g'\)^2x^\mu dx
\end{aligned}
\label{CKNe}
 \end{equation}
 holds for all $g\in \mathcal{D}_\mu(\mathbb{R}_+)$ with the best constant factor on the right-hand side. 
Moreover, all the functions achieving the best constant is given by the set 
\[
\left\{C e^{-\lambda x}{}_1F_1(b,\mu,\lambda x)\ ;\ C\in\mathbb{R},\ \lambda>0\right\},
\]}%
where ${}_1F_1(b,\mu,\cdot)$ is  Kummer's confluent hypergeometric function (see e.g. \cite{Abramowitz1972}) given by
 \begin{equation}
  {}_1F_1(b,\mu,x)=\sum_{k=0}^\infty 
   \frac{(b)_k}{(\mu)_k} 
   \frac{x^k}{k!}
   \quad\text{ with}\quad b=\frac{\mu-\sqrt{\mu^2-4\:\!{\varepsilon} }}{2},
\qquad 
\label{1F1}
 \end{equation}
where \red{$(q)_k=\left\{\begin{array}{ll}
1&(k=0)	    
\\ q(q+1)(q+2)\cdots (q+k-1)&(k\in\mathbb{N})\end{array}\right.
$ denotes the rising factorial.} 
\end{theorem}
\begin{proof}
{\it Step 1}. \red{
Let us apply to $g\in \mathcal{D}_\mu(\mathbb{R}_+)\setminus\{0\}$ the same notations in \eqref{Rg} and \eqref{P0P2Q}, and we additionally set\[  
 \widetilde{R}[g]=\dfrac{Q[g]+P_1[g]}{P_0[g]}.\]
Here we show that the minimization problem for $R[g]=Q[g]P_1[g]/(P_0[g])^2$ can be reduced to that of $\widetilde{R}[g]$; the author of the present paper knew this idea for the first time in the recent paper \cite{CKN_second} by Duong-Nguyen. 
To this end, first of all let us check that
 \begin{equation}
  \tfrac{1}{4}\(\widetilde{{R}}[{g}]\)^2\ge {R}[{g}]%
 \ge \tfrac{1}{4}\inf_{{g}\in\mathcal{D}_\mu}\(\widetilde{R}[{g}]\)^2,%
 \label{RRR}
 \end{equation}
where and hereafter the abbreviation $g\in\mathcal{D}_\mu$ means $g\in\mathcal{D}_\mu(\mathbb{R}_+)\setminus\{0\}$. 
 The first inequality in \eqref{RRR} is easy to verify; to check the second,  
 notice that $R[g]$ is invariant under the scaling transformation \[{g}({ x})\mapsto {g}_\lambda({ x})={g}(\lambda { x})\] 
 for any $\lambda>0$. 
 By the change of variables ${ y}=\lambda { x}$, we have $\frac{\partial}{\partial x_k}{g}_\lambda({ x})=\lambda \frac{\partial}{\partial y_k}{g}({ y})$, which leads to
 \[
 \begin{split}
  \widetilde{R}[{g}_\lambda]
 &%
 =\frac{\lambda^2 Q[g]
 +\lambda^{-2}P_1[g]}
 {P_0[g]}.
 \end{split}
 \]
 Choosing $\lambda=\(\dfrac{P_1[g]}{Q[g]}\)^{1/4}$, we get $\widetilde{R}[{g}_\lambda]=2\sqrt{{R}[{g}]}$ or equivalently 
 \begin{equation}
 R[g]=\tfrac{1}{4}\(\widetilde{R}[g_\lambda]\)^2.\label{RR}
 \end{equation}
 Combining this equation with $\(\widetilde{R}[g_\lambda]\)^2\ge \displaystyle\inf_{g\in\mathcal{D}_\mu}\(\widetilde{R}[g]\)^2$, we get the second inequality of \eqref{RRR} as desired. Taking the infimum on both sides of \eqref{RRR} then yields
 \begin{equation}
 \inf_{g\in\mathcal{D}_\mu}R[g]=\tfrac{1}{4}\inf_{g\in\mathcal{D}_\mu}\(\widetilde{R}[g]\)^2.
 \label{infR}
 \end{equation}
 Next, we derive the relation between the minimizers of $R[g]$ and $\widetilde{R}[g]$. Assume that $g_0\in\mathcal{D}_\mu$ is a minimizer of $R[g]$:
 \[
 R[g_0]=\tfrac{1}{4}\inf_{g\in\mathcal{D}_\mu} R[g].
 \]
Then, noticing  from \eqref{RR} that $R[g_0]=\tfrac{1}{4}\(\widetilde{R}[(g_0)_\lambda]\)^2$ for some $\lambda>0$, we find from \eqref{infR} that
 \[
 \tfrac{1}{4}\(\widetilde{R}[(g_0)_\lambda]\)^2 =\tfrac{1}{4}\inf_{g\in\mathcal{D}_\mu}\(\widetilde{R}[g]\)^2,
 \]
 whence
 \[
 \widetilde{R}[(g_0)_\lambda]= \inf_{g\in\mathcal{D}_\mu}\widetilde{R}[g].
 \]
In other words, we conclude that, if ${g}_0\in\mathcal{D}_\mu$ is a minimizer of $R[g]$ then $(g_0)_\lambda$ is that of $\widetilde{R}[g]$ for {\it some} $\lambda>0$. Conversely, if ${g}_0\in\mathcal{D}_\mu$ is a minimizer of $\widetilde{R}[g]$ then  $(g_0)_\lambda$ is that of $R[g]$ for {\it all} $\lambda>0$. }

\red{{\it Step 2.}  We start  with the nonnegative functional
 \begin{equation}
  I[{g}]:=\int_0^\infty\Big(x{g}''+(x+\mu){g}'+(\mu-b){g}\Big)^2x^{\mu-1}dx
 \label{I[g]}
 \end{equation}
for ${g}\in C_\mu^\infty(\mathbb{R}_+)$. 
 By expanding the integrand, let $I[{g}]$ split into %
 \[
  I[{g}]=I_0[{g}]+I_1[{g}],
 \]
 where
 \[
 \begin{aligned}
  I_0[{g}]&:=\int_0^\infty\(
 x^2\({g}''\)^2+(x+\mu)^2\(g'\)^2+(\mu-b)^2{g}^2\)x^{\mu-1}dx,
 \\
 I_1[{g}]&: 
 =2\int_0^\infty\(
  x(x+\mu){g}''g'
 +(\mu-b)\Big(x{g}'' {g}+(x+\mu){g}'{g}\Big)
 \)x^{\mu-1}dx.
 \end{aligned}
 \]
 To compute $I_1[g]$ by integration by parts, we have
 \[
 \begin{split}
 I_1[{g}]%
 &=\int_0^\infty
 \(
 \begin{aligned}
 & \(x^{\mu+1}+\mu x^\mu\)\((g')^2\)'
 \\&+(\mu-b)
 \(2\:\! x^{\mu}g''g+\(x^\mu+\mu x^{\mu-1}\)\(g^2\)'\)
 \end{aligned}
 \)dx
 \\&=\int_0^\infty \frac{d}{dx}\(\(x^{\mu+1}+\mu x^\mu\)(g')^2+(\mu-b)\(
 \begin{aligned}
 2\:\! x^\mu g'g-\mu x^{\mu-1}g^2
 \\
 +(x^\mu+\mu x^{\mu-1})g^2
 \end{aligned}
 \)\)dx
 \\&\quad +\int_0^\infty
 \(
 \begin{aligned}
 &-\((\mu+1)x^\mu+\mu^2x^{\mu-1}\)(g')^2
 \\&
 +(\mu-b)\(
 \begin{aligned}
 &\mu(\mu-1)x^{\mu-2}g^2-2(g')^2x^\mu
 \\&-\(\mu x^{\mu-1}+\mu(\mu-1)x^{\mu-2}\)g^2
 \end{aligned}
 \)
 \end{aligned}\)dx
 \\&=%
 \int_0^\infty
 \Big(
 \((2b-3\mu-1)x-\mu^2\)(g')^2-(\mu-b)\mu g^2
 \Big)x^{\mu-1}dx,
 \end{split}
 \]
 where the second equality follows with the aid of the identity
 \[
 \begin{aligned}
 x^\mu {g}''{g}
 &=\frac{d}{dx}\(x^\mu g'g-\frac{\mu}{2}x^{\mu-1}{g}^2\)+ \(\frac{\mu(\mu-1)}{2\:\!x^2}{g}^2-({g}')^2\)x^\mu 
 \end{aligned}
 \]
 and the third follows from the condition \eqref{boundary_f} together with the cancellation of the terms $\mu x^{\mu-1}{g}^2$.  
 Then we get
 \[
 \begin{aligned}
  I_0[g]+I_1[g]
 &=\int_0^\infty\(
 \begin{aligned}
 &x^2(g'')^2+\((x+\mu)^2+(2b-3\mu-1)x-\mu^2\)(g')^2
 \\&+\((\mu-b)^2-(\mu-b)\mu\) g^2
 \end{aligned}
 \)x^{\mu-1}dx
 \\&=\int_0^\infty
 \Big(\begin{aligned}
   x^2(g'')^2+\(x^2-(\mu+1-2b)x\)(g')^2-b(\mu-b)g^2
  \end{aligned}
 \Big)x^{\mu-1}dx.
 \\&=\int_0^\infty (g'')^2x^{\mu+1}dx+\int_0^\infty\(x^2(g')^2-\varepsilon g^2\)x^{\mu-1}dx
  \\&\quad -\(\sqrt{\mu^2-4\:\!\varepsilon}+1\)\int_0^\infty (g')^2x^\mu dx
 \end{aligned}
 \]
 from the definition of $b$. 
 Therefore, we have obtained the identity
 \begin{align}
 I[g]
 = Q[g]+P_1[g]-c P_0[g]\qquad
 \text{with}\quad 
c:=\sqrt{\mu^2-4\:\!\varepsilon}+1
 \label{identity_I}
 \end{align}
 for all ${g}\in C_\mu^\infty(\mathbb{R}_+)$;  the same also applies to $g\in\mathcal{D}_\mu(\mathbb{R}_+)$ by the density argument. In particular, it is clear from the non-negativity of \eqref{I[g]} that 
 \begin{equation}
  Q[g]+P_1[g]\ge cP_0[g]
   \qquad\text{ or }\ 
   \widetilde{R}[g]\ge c
\label{Q+P1>cP0}
 \end{equation} 
 for all ${g}\in \mathcal{D}_\mu(\mathbb{R}_+)$, where the equality holds if and only if $I[g]=0$ or equivalently
 \begin{equation}
  x{g}''+(x+\mu){g}'+(\mu-b){g}=0
 \qquad\text{on}\ \mathbb{R}_+.
 \label{ODE_f}
 \end{equation}
In order to specify the equality condition for \eqref{Q+P1>cP0} or to solve the equation \eqref{ODE_f},  let us take the transformation
\begin{equation}
g(x)=e^{-x} w(x)
\label{g=ew}
\end{equation}
which changes \eqref{ODE_f} into
 \[ \begin{aligned}
  x&(e^{-x} w)''+(x+\mu)(e^{-x} w)'+(\mu-b)e^{-x} w
  \\&=x\(e^{-x}w-2e^{-x}w'+e^{-x}w''\)+(x+\mu)\(-e^{-x} w+e^{-x}w'\)+(\mu-b)e^{-x} w
  \\&=e^{-x}\(xw''+(\mu-x)w'-b w\)=0.
 \end{aligned}
 \]
Then we get the Kummer equation  
 \[
 xw''+(\mu-x)w'-b w=0
 \]
 which is known to have the two independent solutions (see e.g. \cite{Abramowitz1972}): 
\[
 \varphi(x)={}_1F_1(b,\mu,x)\quad\text{ and }\quad
\psi(x)=x^{1-\mu}{}_1F_1(b+1-\mu,\,2-\mu,\,x).
\]
Notice that the function $e^{-x}\varphi(x)$ belongs to $\mathcal{D}_\mu(\mathbb{R}_+)$, while $e^{-x}\psi(x)$ does not since 
 $\int_0^\infty\((e^{-x}\psi(x))''\)^2x^{\mu+1}dx=\infty$ from  $\psi''(x)=O(x^{-1-\mu})$ as $x\to0$. Therefore, $e^{-x}\varphi(x)$ is the unique solution to \eqref{ODE_f} (up to constant multiplication) when restricted to the space $\mathcal{D}_\mu(\mathbb{R}_+)$.  
 By recalling Step 1, we have obtained that
 \begin{equation}
 \tfrac{1}{4}\min_{g\in\mathcal{D}_\mu}\(\widetilde{R}[g]\)^2=\min_{g\in\mathcal{D}_\mu} R[g]=c^2/4,
 \label{est_R}
 \end{equation}
 and that all the minimizers of $\widetilde{R}[g]$ are given by the set
\[
 \left\{Ce^{-x}{}_1F_1(b,\mu,x)\ ;\ C\in\mathbb{R}\right\}
\]
which is also a subset of the minimizers of $R[g]$.
}

\red{
 Finally, 
let us specify the minimizers $g_0$ of $R[g]$. 
 According to the conclusion of Step 2, there exists some $\lambda>0$ such that $(g_0)_\lambda(x)={g}_0(\lambda x)$ is a minimizer of $\widetilde{R}[g]$. Then it follows from the above discussion that 
 \[
 (g_0)_\lambda(x)= {g}_0(\lambda x)\propto e^{-x}{}_1F_1(b,\mu,x)\qquad\text{for some }\lambda>0
 \]
 or equivalently that 
 \[
 g_0(x)\propto e^{-\lambda x}{}_1F_1(b,\mu,\lambda x)
 \qquad\text{for some }\lambda>0.
 \]
Therefore, it turns out that the set of  all minimizers of $R[g]$ is given by
 \[
 \left\{C e^{\lambda x} {}_1F_1(b,\mu,\lambda x)\ ; \ C\in\mathbb{R},\ \lambda>0\right\}.
 \]
 The proof of the theorem is now complete.}
\end{proof}

\red{For later use, the following fact provides a sufficient condition ensuring that both the two cases $g\in C_c^\infty(\mathbb{R}_+)$ and $g\in\mathcal{D}_\mu(\mathbb{R}_+)$ give the same infimum value of $R[g]$:}
\begin{lemma}
\label{lemma:density_1}
If $\mu>2$, then $C_c^\infty(\mathbb{R}_+)$ is dense in $\mathcal{D}_\mu(\mathbb{R}_+)$.
\end{lemma}
\begin{proof}
Let $g=g(x)\in \mathcal{D}_\mu(\mathbb{R}_+)\cap C_\mu^\infty(\mathbb{R}_+)$ and
let $\zeta_0=\zeta_0(x)\in C_c^\infty(\mathbb{R})$ with $\zeta_0(0)=1$. Define $\{g_n\}_{n\in\mathbb{N}}\subset C_c^\infty(\mathbb{R}_+)$ by the formula
\[
 g_n(x)=\zeta_n(x)g(x),\quad\text{where}\quad \zeta_n(x)=\zeta_0\(\tfrac{1}{n}\log x\)
\quad(\forall n\in\mathbb{N},\ \forall x>0).
\]
In order to see that $g_n\to g$ in $\mathcal{D}_\mu(\mathbb{R}_+)$,  a direct calculation yields
\[
\begin{aligned}
 & \red{{g}'_n}-\zeta_n \red{{g}'}=(\partial\zeta_n)\frac{g}{x},
\\
&\red{g''_n}-\zeta_n\red{{g}''}=(\partial^2\zeta_n-\partial\zeta_n)\frac{g}{x^2}+2(\partial\zeta_n) \frac{\red{{g}'}}{x},
\end{aligned}
\]
\red{where $\partial=x \frac{d}{dx}$ as  in \eqref{abb}}. 
Taking the $L^2$ integral of them yields
\[
\begin{aligned}
& \int_0^\infty\(\red{{g}_n'}-\zeta_n\red{{g}'}\)^2(x+1)x^\mu dx=O(n^{-2})\int_{0}^\infty g^2\(1+\frac{1}{x}\)x^{\mu-1} dx,
\\
& \int_0^\infty\(g_n''-\zeta_n \red{g''}\)^2x^{\mu+1}dx=O(n^{-2})\int_0^\infty\(\frac{g^2}{x^2}+(\red{{g}'})^2\)x^{\mu-1}dx.
\end{aligned}
\]
Notice from \eqref{boundary_f} that $g(x)$ is bounded near $x=0$; then it follows from $\|g\|_{\mathcal{D}_\mu}<\infty$ that both the integrals on the right-hand sides are finite as far as $\mu>2$. Since the convergences $\zeta_n\red{{g}'}\to \red{{g}'}$ and $\zeta_n\red{{g}''}\to \red{g''}$ are dominated, by way of the $L^2$ triangle inequality, we consequently get
\[
 \begin{aligned}
  &\int_0^\infty(\red{{g}'_n}-\red{g'})^2(x+1)x^\mu dx\to 0
\\
\text{and}\quad & \int_0^\infty\(\red{{g}''_n}- \red{g''}\)^2 x^{\mu+1}dx\to0
 \end{aligned}
\]
as well as $\int_0^\infty \(g_n-g\)^2x^{\mu-1}dx\to0$ $(n\to\infty)$. 
This fact means $\|g_n-g\|_{\mathcal{D}_\mu}\to0$, as desired.
\end{proof}

\section{Conclusion}
\label{sec:concl}
For every $\nu\in\mathbb{N}$, we set
\[
\mu_\nu= \frac{N}{2}+\nu
\quad\text{ and }\quad
\varepsilon_\nu=\frac{\nu+N-2}{2}.
\]
This setting is the same as in \eqref{mu_and_ep} with $(\mu,\varepsilon)=(\mu_\nu,\varepsilon_\nu)$, which clearly satisfies both the conditions $\mu>2$ in Lemma~\ref{lemma:density_1} and \red{$0<\varepsilon<\mu^2/4$} in Theorem~\ref{theorem:3}.  By using their results, it holds from \eqref{C_PNn} that
\begin{align}
 C_{P,N,\nu}&=\red{4}\inf_{g\in C_c^\infty(\mathbb{R}_+)}R[g]=\red{4}\inf_{g\in\mathcal{D}_\mu} R[g]
=\(\sqrt{\mu_\nu^2-4\:\!\varepsilon_\nu}+1\)^2
\\&=\mfrac{1}{4}\(\sqrt{(2\nu+N-2)^2-4(N-3)}+2\)^2,
\end{align}
from which we directly find that $C_{P,N,\nu+1}>C_{P,N,\nu}$ for all $\nu\in\mathbb{N}$. Then we get from \eqref{C_PN} that
\[
 C_{P,N}=C_{P,N,1}=\(\sqrt{\mu_1^2-4\:\!\varepsilon_1}+1\)^2=\mfrac{1}{4}\(\sqrt{N^2-4(N-3)}+2\)^2.
\]
Comparing this constant with that in Theorem~\ref{theorem:tor}, we see from \eqref{C_N} that 
\begin{align}
 C_N&=\min\left\{\tfrac{1}{4}\(\sqrt{N^2-4(N-3)}+2\)^2, \ \tfrac{1}{4}\(N+2\)^2\right\}
\\&=\mfrac{1}{4}\(\sqrt{N^2-4(N-3)}+2\)^2
=\begin{cases}
C_{P,3}=C_{T,3}&(N=3)
\\
 C_{P,N}<C_{T,N} &(N\ge4)
 \end{cases},
\label{CN}
\end{align}
which gives the desired best HUP constant in Theorem~\ref{theorem}. 

For the completeness of the proof of Theorem~\ref{theorem}, it remains to specify the profile(s) of extremal solenoidal fields which attain the best HUP constant. To this end, let ${\bm u}_0\in  C^\infty(\dot{\mathbb{R}}^N)^N$ be the poloidal field of the form 
\begin{equation}
 {\bm u}_0({\bm x})={\bm D}\(\frac{{\bm c}\cdot {\bm x}}{|{\bm x}|}g_0(|{\bm x}|^2)\)
\label{extremal_pol}
\end{equation}
with ${\bm c}\in\mathbb{R}^N\setminus\{{\bm 0}\}$ being any constant vector (which serves as the axis of symmetry of ${\bm u}_0$),
where $g_0\in \mathcal{D}_{\mu_1}(\mathbb{R}_+)$ is given by
\begin{align}
g_0({s})&=\red{e^{-\lambda {s}}{}_1F_1(b_1,{\mu_1},\lambda{s})}
\qquad\text{with } \quad 
b_1=\frac{\mu_1-\sqrt{\mu_1^2-4{\varepsilon_1} }}{2}\quad\text{and}\quad\red{\lambda>0}
\end{align}
as in \eqref{1F1} 
with $(b,\mu,\varepsilon)=(b_1,\mu_1,\varepsilon_1)$. 
We have to check that
\[
{\bm u}_0\in\mathcal{D}(\mathbb{R}^N)
\quad\text{ and }\quad 
\frac{\int_{\mathbb{R}^N}|\nabla {\bm u}_0|^2dx \int_{\mathbb{R}^N}|{\bm u}_0|^2|{\bm x}|^2dx}{\(\int_{\mathbb{R}^N}|{\bm u}_0|^2dx\)^2}
= C_N.
\] 
To do so, define the linear operator ${\bm D}_1:C^\infty(\mathbb{R}_+)\to C^\infty(\dot{\mathbb{R}}^N)^N
$ by
\[
 \({\bm D}_1 g\)({\bm x})={\bm D}\(\frac{{\bm c}\cdot {\bm x}}{|{\bm x}|}g(|{\bm x}|^2)\)\qquad \forall g\in C^\infty(\mathbb{R}_+),%
\]
in order that ${\bm D}_1g_0={\bm u}_0$. 
Applying \eqref{f=gY} and \eqref{Df_Rg} to the case $\nu=1$ and $Y({\bm \sigma})={\bm c}\cdot {\bm \sigma}$,  we then see that the  equations
\begin{equation}
\left.\begin{gathered}
c\int_{\mathbb{R}^N}|{\bm D}_1g|^2|{\bm x}|^{2\beta}=P_\beta [g],\quad  c\int_{\mathbb{R}^N}|\nabla {\bm D}_1g|^2dx=\red{4}\:\!Q[g]
\\\text{and}\quad  \frac{\int_{\mathbb{R}^N}|\nabla {\bm D}_1 g|^2dx \int_{\mathbb{R}^N}|{\bm D}_1 g|^2|{\bm x}|^2dx}{\(\int_{\mathbb{R}^N}|{\bm D}_1 g|^2dx\)^2}=\red{4}R[g]
\end{gathered}
\ \right\}\
\forall g\in C_c^\infty(\mathbb{R}_+)
\label{D1_PQR}
 \end{equation}
hold with $c>0$ depending only on $N$, 
where $P_\beta[g]$, $Q[g]$ and $R[g]$ are the same as in \eqref{P0P2Q} and \eqref{Rg}  with $(\mu,\varepsilon)=(\mu_1,\varepsilon_1)$. In particular, we get
\[
c\|{\bm D}_1 g \|_{\mathcal{D}}^2=P_0[g]+P_1[g]+\red{4}\;\! Q[g]\le \|g\|_{\mathcal{D}_{\mu_1}}^2\qquad \forall g\in C_c^\infty(\mathbb{R}_+)
\]
in terms of the notations \eqref{norm_D} and \eqref{norm_Dm}. 
Therefore, by way of the density arguments in Lemmas \ref{lemma:density} and \ref{lemma:density_1}, it holds that
\[
 c\|{\bm D}_1 g\|_{\mathcal{D}}\le \|g\|_{\mathcal{D}_{\mu_1}}\qquad \forall g\in\mathcal{D}_{\mu_1}(\mathbb{R}_+)\cap C^\infty(\mathbb{R}_+),
\]
which says that ${\bm D}_1$ is continuous on $C^\infty(\mathbb{R}_+)$ from the topology of $\mathcal{D}_{\mu_1}(\mathbb{R}_+)$ to $\mathcal{D}(\mathbb{R}^N)$. Hence the same formulae in \eqref{D1_PQR} are still applicable to 
\[\forall g\in\mathcal{D}_{\mu_1}(\mathbb{R}_+)\cap C^\infty(\mathbb{R}_+);\]
in particular, by letting $g=g_0$, we obtain $\|{\bm u}_0\|_{\mathcal{D}}=\|{\bm D}_1g_0\|_{\mathcal{D}}<\infty$ and
\[\frac{\int_{\mathbb{R}^N}|\nabla {\bm u}_0|^2dx\int_{\mathbb{R}^N}|{\bm u}_0|^2|{\bm x}|^2dx}{\(\int_{\mathbb{R}^N}|{\bm u}_0|^2dx\)^2}=\red{4}R[g_0]=\(\sqrt{\mu_1^2-4\:\!\varepsilon_1}+1\)^2=C_N,\]
where the second equality holds from Theorem \ref{theorem:3} with $(\mu,\varepsilon)=(\mu_1,\varepsilon_1)$. This proves the desired result.

In view of \eqref{CN}, for the case $N=3$,  there is another profile of the extremal solenoidal fields:  
 according to Theorem~\ref{theorem:tor}, the best HUP constant $C_3$ is also achieved by the toroidal field
 \begin{equation}
  {\bm u}({\bm x})=\({\bm c}_1\cdot {\bm x},\ {\bm c}_2\cdot {\bm x}, \  {\bm c}_3\cdot {\bm x}\)e^{-c|{\bm x}|^2}
\label{extremal_tor}
 \end{equation}
 with $c>0$ and with $({\bm c}_1,{\bm c}_2,{\bm c}_3)\in\mathbb{R}^{3\times 3}$ being any antisymmetric constant $3\times 3$ matrix. 

The proof of Theorem~\ref{theorem} is now complete.
\section*{\bf Acknowledgments}
\noindent
This work is supported by JSPS KAKENHI Grant number JP21J00172. 
The author thanks to Prof. Y. Kabeya (Osaka Prefecture University) for his great support and encouragement. Additionally, this research was partly supported by JSPS KAKENHI Grant-in-Aid for Scientific Research (B) 19H01800 (F. Takahashi) and Osaka City University Advanced Mathematical Institute (MEXT Joint Usage/Research Center on Mathematics and Theoretical Physics JPMXP0619217849). %

\bibliographystyle{amsplain-it}
\bibliography{bibdata,bibmydata}
\end{document}